\documentclass[11pt]{article}
\usepackage{amsfonts}
\usepackage{mathrsfs}
\usepackage{bbm}
\usepackage{indentfirst}
\usepackage{amssymb,amsmath,amsthm,graphicx}
\usepackage{float}
\usepackage[colorlinks, linkcolor=red]{hyperref}
\usepackage{color, xcolor}
\usepackage{cite}
\usepackage{enumerate}
\usepackage{esint}
\usepackage[paper=a4paper, left=1.6cm, right=1.6cm, top=1.8cm, bottom=1.6cm, headheight=5.5pt, footskip=0.8cm, footnotesep=0.8cm, centering, includefoot]{geometry}

\hypersetup{urlcolor=red, citecolor=red}

\DeclareMathOperator{\divv}{div}
\DeclareMathOperator{\curl}{curl}
\DeclareMathOperator{\loc}{loc}
\DeclareMathOperator{\divf}
{di\overset{\raisebox{0.1ex}{\hspace{0.1em}$\mathbf{\cdot}$}}{v}}
\allowdisplaybreaks

\begin{document}
\title{Global weak solutions and incompressible limit to the isentropic compressible Navier--Stokes equations in 2D bounded domains with ripped density and large initial data
\thanks{
Wu's research was partially supported by Fujian Alliance of Mathematics (No. 2023SXLMMS08) and the Scientific Research Funds of Xiamen University of Technology (No. YKJ25009R).
Zhong's research was partially supported by Fundamental Research Funds for the Central Universities (No. SWU--KU24001) and National Natural Science Foundation of China (No. 12371227).}
}

\author{Shuai Wang$\,^{\rm 1}\,$,\ Guochun Wu$\,^{\rm 2}\,$,\
Xin Zhong$\,^{\rm 1}\,$ {\thanks{E-mail addresses: swang238@163.com (S. Wang),
guochunwu@126.com (G. Wu), xzhong1014@amss.ac.cn (X. Zhong).}}\date{}\\
\footnotesize $^{\rm 1}\,$
School of Mathematics and Statistics, Southwest University, Chongqing 400715, P. R. China\\
\footnotesize $^{\rm 2}\,$ School of Mathematics and Statistics, Xiamen University of Technology, Xiamen 361024, P. R. China}

\maketitle
\newtheorem{theorem}{Theorem}[section]
\newtheorem{definition}{Definition}[section]
\newtheorem{lemma}{Lemma}[section]
\newtheorem{proposition}{Proposition}[section]
\newtheorem{corollary}{Corollary}[section]
\newtheorem{remark}{Remark}[section]
\renewcommand{\theequation}{\thesection.\arabic{equation}}
\catcode`@=11 \@addtoreset{equation}{section} \catcode`@=12
\maketitle{}

\begin{abstract}
This paper is a continuation of our previous work (arXiv:2507.03505), where the global existence and incompressible limit of weak solutions to the isentropic compressible Navier--Stokes equations in the half-plane with ripped density and large initial data were established. We extend such results to the case of two-dimensional bounded convex domains under a Navier-slip boundary condition. To overcome difficulties in the presence of a curved boundary, some new estimates based on the effective viscous flux and a Desjardins-type logarithmic interpolation inequality play decisive roles.
\end{abstract}

\textit{Key words and phrases}. Navier--Stokes equations; global weak solutions; incompressible limit; slip boundary conditions; large initial data; vacuum.

2020 \textit{Mathematics Subject Classification}. 35Q30; 35A01; 35B40.

\tableofcontents

\section{Introduction}
\subsection{Background and motivation}
Assume that $\Omega\subset\mathbb{R}^2$ is a bounded convex domain with smooth boundary $\partial\Omega$, we consider the following isentropic compressible Navier--Stokes equations in $\Omega\times(0,+\infty)$:
\begin{align}\label{a1}
\begin{cases}
\rho_t+\divv(\rho\mathbf{u})=0,\\
(\rho\mathbf{u})_t+\divv(\rho\mathbf{u}\otimes\mathbf{u})-\mu\Delta\mathbf{u}
-(\mu+\lambda)\nabla\divv\mathbf{u}+\nabla P
=0.
\end{cases}
\end{align}
The unknowns $\rho$, $\mathbf{u}=(u^1,u^2)$, and $P=P(\rho)=a\rho^\gamma\ (a>0,\gamma>1)$ are the fluid density, velocity, and pressure, respectively. The constants $\mu$ and $\lambda$ represent the shear viscosity and bulk viscosity of the fluid, respectively, satisfying physical restrictions
\begin{equation*}
\mu>0,\ \ \ \mu+\lambda\geq0.
\end{equation*}

We complement \eqref{a1} with the given initial data
\begin{equation}
(\rho,\mathbf{u})|_{t=0}=(\rho_0,\mathbf{u}_0)(\mathbf{x}),\ \ \mathbf{x}\in\Omega,
\end{equation}
and a Navier-slip boundary condition
\begin{equation}\label{a3}
\mathbf{u}\cdot \mathbf{n}=0,  ~~ \curl\mathbf{u}\triangleq\partial_1u^2-\partial_2u^1=0, \ \ \mathbf{x}\in\partial\Omega,\ t>0,
\end{equation}
where $\mathbf{n}=(n^1,n^2)$ is the unit outward normal vector to $\partial\Omega$.

For $t\geq0$, observe that the solutions to \eqref{a1} adhere to the global energy law
\begin{align*}
\int_{\Omega}\bigg(\frac{1}{2}\rho |\mathbf{u}|^2+G(\rho)\bigg)\mathrm{d}\mathbf{x}
+\int_0^t\int_{\Omega}\big[(2\mu+\lambda)(\divv \mathbf{u})^2+\mu(\curl\mathbf{u})^2\big]\mathrm{d}\mathbf{x}\mathrm{d}\tau
\leq C_0,
\end{align*}
with the initial total energy $C_0$ and the potential energy $G(\rho)$ being given by
\begin{equation}\label{G}
C_0\triangleq\int_{\Omega}
\bigg(\frac{1}{2}\rho_0|\mathbf{u}_0|^2+G(\rho_0)\bigg)\mathrm{d}\mathbf{x},~~
G(\rho)\triangleq\rho\int_{\bar{\rho}}^\rho\frac{P(\xi)-P(\bar{\rho})}{\xi^2}\mathrm{d}\xi,
\end{equation}
where the average of density over $\Omega$, due to the mass equation \eqref{a1}$_1$, is indeed a nonnegative constant
\begin{equation*}
  \bar{\rho}\triangleq\frac{1}{|\Omega|}\int_\Omega \rho\mathrm{d}\mathbf{x}=\frac{1}{|\Omega|}\int_\Omega \rho_0\mathrm{d}\mathbf{x}.
\end{equation*}
Hence, the incompressibility can be recovered formally from the global energy law as $\lambda\rightarrow\infty$.

Significant progress has been made over the past four decades regarding the global existence of solutions for the multi-dimensional isentropic compressible Navier--Stokes equations. In the 1980s, Matsumura and Nishida \cite{MN80,MN83} established the global well-posedness of classical solutions for 3D initial-boundary value problem when the initial data differ only slightly from the equilibrium values (constants). Applying harmonic analysis techniques, especially the Littlewood-Paley theory, Danchin \cite{Da00} demonstrated the global existence and uniqueness of strong solutions in the framework of critical Besov space $\dot{B}^{\frac{n}{2}-1}_{2,1}(\mathbb{R}^n)\ (n\geq2)$. Later on, Charve and Danchin \cite{CD10}, and independently by Chen--Miao--Zhang \cite{CCZ10}, extended Danchin's work \cite{Da00} to $\dot{B}^{\frac{n}{p}-1}_{p,1} (\mathbb{R}^n)\ (p\in[2,2n),n\geq2)$, which is also a critical framework but is not related to the energy space (see also \cite{H11} for a related work). It should be emphasized that all results stated above do not allow the presence of vacuum states (i.e., the density may vanish in some open sets).
There are also some interesting results concerning global regular solutions with non-vacuum for a class of large initial data, please refer to \cite{DM17,FZZ18,HHW19,ZLZ20} for further details.

However, as pointed out by many authors (see, e.g., \cite{DM23,HLX12,M1,M2}), the possible presence of vacuum is one of the major difficulties in the study of mathematical theory of compressible fluids due to the high singularity and degeneracy of the system \eqref{a1} near the vacuum region. The major breakthrough on the global solutions with vacuum is due to P.-L. Lions \cite{PL98}, where he employed the renormalization skills to establish global weak solutions in $\mathbb{R}^n$ for $\gamma\geq\frac{3n}{n+2}\ (n=2,3)$. Feireisl--Novotn\'y--Petzeltov\'a \cite{EF01} later extended Lions' result to the case $\gamma>\frac{n}{2}$ by introducing oscillation defect measure, while $\gamma=1$ for the 2D Dirichlet problem was obtained in \cite{PW15}. Moreover, Jiang and Zhang \cite{JZ01,JZ03} addressed global weak solutions in three dimensions for any $\gamma>1$ when the initial data are assumed to have some spherically symmetric or axisymmetric properties. A key issue in \cite{JZ01,JZ03} is to apply properties of the effective viscous flux to obtain extra information about integrability of the density. Nonetheless, due to the possible concentration of finite kinetic energy in very small domains \cite{H21}, it still seems to be a challenge to show the global existence of weak solutions with general 3D initial data for $\gamma \in (1,\frac{3}{2}]$. Furthermore, Bresch and Jabin \cite{BJ18} investigated the global existence of weak solutions to the compressible Navier--Stokes equations with general viscous stress tensor. In particular, no monotonicity assumptions on the pressure law are required in \cite{BJ18}. In spite of the enormous amount of effort, the question of global regularity and uniqueness for such weak solutions remains completely open.

Meanwhile, several results are devoted to investigating the global well-posedness of strong (or classical) solutions to the isentropic compressible Navier--Stokes equations with vacuum. Huang--Li--Xin \cite{HLX12} established the global existence and uniqueness of classical solutions for system \eqref{a1} in $\mathbb{R}^3$ with smooth initial data which are of small total energy but possibly large oscillations, where the far-field density could be either vacuum or non-vacuum. In \cite{HHPZ24}, the initial total energy could be large when the adiabatic exponent $\gamma$ is sufficiently close to $1$. Using some key \textit{a priori} decay with rates (in large time) and a spatially weighted energy method, Li and Xin \cite{LX19} dealt with the case of the whole plane with vacuum at infinity provided that the initial total energy is suitably small. Recently, the Cauchy problem of \cite{HLX12} was extended to the Navier-slip boundary problem in \cite{CAI23}. It should be noticed that a central aspect of their studies in \cite{CAI23,HLX12,LX19,HHPZ24} is the derivation of both the time-independent upper bound of the density and time-dependent higher-order estimates of smooth solutions.

Apart from {\it large-energy weak solutions} \cite{PL98,EF01} and {\it small-energy classical solutions} \cite{HLX12,LX19}, the third type of solutions to \eqref{a1} is the so-called {\it Hoff's intermediate weak
solutions}. More precisely, Hoff \cite{Hoff95,Hoff95*,Hoff02,Hoff05} proved the global existence of such solutions under certain additional assumptions on the initial data. Indeed, the regularity of weak solutions constructed by Hoff \cite{Hoff95,Hoff95*,Hoff02,Hoff05} is stronger than Lions--Feireisl's solutions in the sense that particle path can be defined in the non-vacuum region, but weaker than the usual strong solutions in the sense that discontinuity of the density can be transported along the particle path. Based on some delicate estimates of the effective viscous flux via Green's function, Perepelitsa \cite{P3} proved the global existence of Hoff's weak solutions in the half-space $\mathbb{R}^n_+\ (n=2,3)$ with no-slip boundary conditions when the initial data are close to a static equilibrium. A further progress was obtained by Hu--Wu--Zhong \cite{HWZ}, where the global existence of weak solutions in \cite{Hoff95,Hoff95*} was improved to the case of bounded non-negative initial density that are of small initial energy but possibly large oscillations. In particular, such solutions converge globally in time to a global weak solution of the inhomogeneous incompressible Navier--Stokes equations as the bulk viscosity tends to infinity. For more studies on the compressible Navier--Stokes equations, one may refer to the excellent handbook \cite{GN18} and references contained therein.

It is certainly interesting to investigate the global existence of regular solutions for multi-dimensional compressible Navier--Stokes equations with {\it large initial data and vacuum}, which is important physically and mathematically. A significant contribution in this regards was made by Danchin and Mucha \cite{DM23}, where the global existence of large solutions with vacuum of \eqref{a1} in $\mathbb{T}^2$ was derived provided the bulk viscosity coefficient is properly large. As a by-product, they also gave a rigorous justification of the convergence to the inhomogeneous incompressible Navier--Stokes equations when the bulk viscosity tends to infinity.
More recently, on the basis of some ideas from \cite{HWZ} and an adaptation of the Desjardins-type logarithmic interpolation inequality in \cite{D1997}, the authors of the present paper \cite{WWZ} extended the results in \cite{HWZ} to the half-plane under a slip boundary condition. It should be emphasized that the absence or flatness of boundary plays a crucial role in \cite{DM23,HWZ,WWZ}. Naturally, one would like to ask: is it possible to establish the global solutions with large initial data and vacuum in a region subject to a {\it curved boundary}?

The primary objective of the current paper is to provide an affirmative answer to this issue. More precisely, we shall address the initial-boundary value problem \eqref{a1}--\eqref{a3} involving general arbitrary large initial data with merely nonnegative bounded density: we recall that a ``ripped'' initial density is a function that may have nontrivial regions of vacuum, without any extra regularity assumption (see \cite{DM23}). Furthermore, we show the limiting behavior of such solutions as $\lambda\rightarrow\infty$ to a solution of the following inhomogeneous incompressible Navier--Stokes equations in $\Omega\times(0,+\infty)$:
\begin{align}\label{a5}
\begin{cases}
\varrho_t+{\bf v}\cdot\nabla\varrho=0,\\
\varrho{\bf v}_t+\varrho{\bf v}\cdot\nabla {\bf v}+\nabla \Pi-\mu\Delta {\bf v}=0,\\
\divv{\bf v}=0,\\
(\varrho,{\bf v})\big |_{t=0}=(\rho_0,{\bf v}_0),
\end{cases}
\end{align}
where ${\bf v}_0$ is the Leray-Helmholtz projection of ${\bf u}_0$ on divergence-free vector fields.

\subsection{Main results}
Before stating our main results, we first formulate the notations and conventions used throughout this paper. We denote by $C$ a generic positive constant which may vary at different places. The symbol $\Box$ denotes the end of a proof and $a\triangleq b$ means $a=b$ by definition. For $1\le p\le \infty$ and integer $k\ge 0$, we denote the standard Sobolev spaces as follows:
\begin{align*}
L^p=L^p(\Omega),\ \ W^{k, p}=W^{k, p}(\Omega),\ \ H^k=W^{k, 2}, \ \
H_\omega^k=\{\mathbf{f}\in H^k:(\mathbf{f}\cdot \mathbf{n})|_{\partial\Omega}=\curl \mathbf{f}|_{\partial\Omega}=0\}.
\end{align*}
For any $f\in L^1_{\loc}(\Omega)$, we define its mollification by $[f]_\epsilon\triangleq j_\epsilon*f$, where $j_\epsilon=j_\epsilon({\bf x})$ is the standard mollifier with width $\epsilon$.
For two $n\times n$ matrices $A=\{a_{ij}\}$ and $B=\{b_{ij}\}$, the symbol $A: B$ represents the trace
of $AB$, that is,
\begin{equation*}
A:B\triangleq\operatorname{tr}(AB)=\sum_{i,j=1}^na_{ij}b_{ji}.
\end{equation*}
In addition, for $\alpha\in (0,1]$, the H\" older seminorm of a function ${\bf v}:U\subseteq\overline{\Omega}\rightarrow \mathbb R^2$ is defined by
\begin{align*}
\langle {\bf v}\rangle^\alpha_{U}
=\sup\limits_{\substack{{\bf x}, {\bf y}\in U\\ {\bf x}\neq {\bf y}}}
\frac{|{\bf v}({\bf x})-{\bf v}({\bf y})|}{|{\bf x}-{\bf y}|^\alpha}.
\end{align*}

Moreover, we write
\begin{align*}
\int f \mathrm{d}\mathbf{x}=\int_{\Omega} f \mathrm{d}\mathbf{x}, ~~\bar{f}\triangleq\fint f\mathrm{d}\mathbf{x}=\frac{1}{|\Omega|}\int_\Omega f\mathrm{d}\mathbf{x}.
\end{align*}
The material derivative of $f$ and the transpose gradient are given respectively by
\begin{equation*}
\dot{f}\triangleq f_t+\mathbf{u}\cdot\nabla f,
~~\nabla^\bot \triangleq(\partial_2,-\partial_1),
\end{equation*}
and the effective viscous flux $F$ is denoted by
\begin{equation}\label{z1.4}
F\triangleq(2\mu+\lambda)\divv\mathbf{u}-(P-\bar{P}).
\end{equation}
Finally, we introduce the Leray projector
\begin{equation*}\label{1.8}
\mathcal{P}\triangleq \text{Id}+\nabla(-\Delta)^{-1}\divv,
\end{equation*}
which projects onto the subspace of divergence-free vector fields inheriting the no-penetration boundary condition, along with its complement $\mathcal{Q}\triangleq \text{Id}-\mathcal{P}$. Both $\mathcal{P}$ and $\mathcal{Q}$ are bounded operators on $L^p$ for any $1<p<\infty$.

We recall the definition of weak solutions to the problem \eqref{a1}--\eqref{a3} in the sense of \cite{Hoff05,NS04}.
\begin{definition}\label{d1.1}
A pair $(\rho, \mathbf{u})$ is said to be a weak solution to the initial-boundary value problem \eqref{a1}--\eqref{a3} provided that
\begin{equation*}
  \rho\in C([0,\infty);H^{-1}(\Omega)),\ \
  \rho\mathbf{u}\in C([0,\infty);\widetilde{H}^{1}(\Omega)^*),\ \
  \nabla\mathbf{u}\in L^2(\Omega\times(0,\infty))
\end{equation*}
with $(\rho,\mathbf{u})|_{t=0}=(\rho_0,\mathbf{u}_0)$, where $\widetilde{H}^{1}(\Omega)^*$ is the dual of $\widetilde{H}^{1}(\Omega)=\{\mathbf{f}\in H^1:(\mathbf{f}\cdot \mathbf{n})|_{\partial\Omega}=0\}$. Moreover,
for any $t_2\geq t_1\geq 0$ and any test function $(\phi,\boldsymbol\psi)(\mathbf{x},t)\in C^1\big(\overline{\Omega}\times[t_1,t_2]\big)$, with uniformly bounded
support in $\mathbf{x}$ for $t\in[t_1,t_2]$ and satisfying $(\boldsymbol\psi\cdot\mathbf{n})|_{\partial\Omega}=0$, the following identities hold\footnote{Throughout this paper, we will use the Einstein summation over repeated indices convention.}:
\begin{align*}
\int_{\Omega}\rho(\mathbf{x},\cdot)\phi(\mathbf{x},\cdot)
\mathrm{d}\mathbf{x}\Big|_{t_1}^{t_2}
&=\int_{t_1}^{t_2}\int_{\Omega}(\rho\phi_t+
\rho\mathbf{u}\cdot\nabla\phi)\mathrm{d}\mathbf{x}\mathrm{d}t,\\
\int_{\Omega}(\rho\mathbf{u}\cdot\boldsymbol{\psi})
(\mathbf{x},\cdot)\mathrm{d}\mathbf{x}\Big|_{t_1}^{t_2}
&=\int_{t_1}^{t_2}
\int_{\Omega}\big(\rho\mathbf{u}\cdot\boldsymbol{\psi}_t+
\rho u^i\mathbf{u}\cdot\partial_i\boldsymbol{\psi}+P\divv\boldsymbol{\psi}\big)
\mathrm{d}\mathbf{x}\mathrm{d}t\notag\\
&\quad -\int_{t_1}^{t_2}
\int_{\Omega}\big((2\mu+\lambda)\divv\mathbf{u}\divv\boldsymbol{\psi}+\mu\curl\mathbf{u}\curl\boldsymbol{\psi}\big)\mathrm{d}\mathbf{x}\mathrm{d}t.
\end{align*}
\end{definition}
Concerning the initial data $(\rho_0,\mathbf{u}_0)$, we always assume that there exist two positive constants $\hat{\rho}$ and $M$ (not necessarily small) such that
\begin{gather}\label{z1.7}
0\leq\inf\rho_0\leq\sup\rho_0\leq\hat{\rho},
\\ \label{z1.8}
\mathbf{u}_0\in H_\omega^1,\  C_0+(2\mu+\lambda)\|\divv\mathbf{u}_0\|_{L^2}^2+\mu\|\curl\mathbf{u}_0\|_{L^2}^2\leq M.
\end{gather}

Now we state our first result on the global existence of weak solutions.

\begin{theorem}\label{t1.1}
Let \eqref{z1.7} and \eqref{z1.8} be satisfied, there exists a positive number $D$ depending only
on $\hat{\rho}$, $a$, $\gamma$, $\mu$, and $\Omega$ such that if
\begin{equation}\label{lam}
\lambda\geq\exp\bigg\{(2+M)^{e^{D(1+C_0)^2}}\bigg\},
\end{equation}
then the problem \eqref{a1}--\eqref{a3} admits a global weak solution $(\rho,\mathbf{u})$ in the sense of Definition $\ref{d1.1}$ satisfying
\begin{equation}\label{reg}
\begin{cases}0\leq\rho(\mathbf{x},t)\leq2\hat{\rho}~a.e.~\mathrm{on}~\Omega\times[0,\infty),\\
(\rho,\sqrt{\rho}\mathbf{u})\in C([0,\infty);L^2(\Omega)),~\mathbf{u}\in L^\infty(0,\infty;H^1(\Omega)),\\
(\nabla^2\mathcal{P}\mathbf{u},\nabla F,\sqrt{\rho}\dot{\mathbf{u}})\in L^2(\Omega\times(0,\infty)),\\
\sigma^{\frac{1}{2}}\sqrt{\rho}\dot{\mathbf{u}}\in L^\infty(0,\infty;L^2(\Omega)),~ \sigma^{\frac{1}{2}}\nabla\dot{\mathbf{u}}\in L^2(\Omega\times(0,\infty)),
\end{cases}
\end{equation}
where $\sigma\triangleq\min\{1,t\}$.
\end{theorem}

The next result will treat the incompressible limit (characterised by the large value of the bulk viscosity) of the global weak solutions established in Theorem \ref{t1.1}.

\begin{theorem}\label{t1.2}
Let $\{(\rho^\lambda,{\bf u}^\lambda)({\bf x},t)\}$ be the family of solutions obtained in Theorem \ref{t1.1}. Then, there exists a subsequence $\{\lambda_k\}$
with $\lambda_k\rightarrow\infty$ such that
\begin{align}\label{1.11}
 \rho^{\lambda_{k}}\rightarrow \varrho ~~&\text{strongly in} ~ L^2(K), \ \
 \text{for any compact set} \ K \subset \Omega \text{ and any} \ t\ge 0 ,\\
 {\bf u}^{\lambda_{k}}\rightarrow {\bf v}
 ~~&\text{uniformly on compact sets in}~\Omega\times(0,\infty),\notag
\end{align}
where $(\varrho,{\bf v})$ is a global weak solution to the inhomogeneous incompressible Navier-Stokes equations \eqref{a5} in the sense of Definition \ref{d1.2} below.
\end{theorem}

\begin{definition}\label{d1.2}
A pair $(\varrho, {\bf v})$ is said to be a weak solution to the problem \eqref{a5} provided that
\begin{gather}
\varrho\in L^\infty(\Omega\times (0,\infty)),~~
\sqrt{\varrho}\mathbf{v}\in L^\infty([0,\infty); L^2(\Omega)),~~
({\bf v},\nabla{\bf v})\in L^2(\Omega\times (0,\infty)),\notag\\
\varrho\in C([0,\infty);L^{2}(\Omega)).\label{1.12}
\end{gather}
Moreover, for any $t_2\geq t_1\geq 0$ and any $C^1$ test function $(\phi,\boldsymbol\psi)$
just as in Definition \ref{d1.1}, which additionally satisfies $\divv\boldsymbol\psi(\cdot,t)=0$ on $\Omega\times[0,\infty)$, the following identities hold:
\begin{gather}
\int_{\Omega}\varrho(\mathbf{x},\cdot)\phi(\mathbf{x},\cdot)\mathrm{d}\mathbf{x}
\Big|_{t_1}^{t_2}=\int_{t_1}^{t_2}\int_{\Omega}(\varrho\phi_t+
\varrho\mathbf{v}\cdot\nabla\phi)\mathrm{d}\mathbf{x}\mathrm{d}t, \label{1.13} \\
\int_{\Omega}(\varrho\mathbf{v}\cdot\boldsymbol{\psi})
(\mathbf{x},\cdot)\mathrm{d}\mathbf{x}\Big|_{t_1}^{t_2}
=\int_{t_1}^{t_2}
\int_{\Omega}\big(\varrho\mathbf{v}\cdot\boldsymbol{\psi}_t+
\varrho v^i\mathbf{v}\cdot\partial_i\boldsymbol{\psi}
-\mu\curl\mathbf{v}\curl\boldsymbol{\psi}\big)\mathrm{d}\mathbf{x}\mathrm{d}t.\label{1.14}
\end{gather}
\end{definition}

Several remarks are in order.

\begin{remark}
It should be noted that Theorem \ref{t1.1} holds for arbitrarily large initial energy as long as the bulk viscosity coefficient is suitably large, which is in sharp contrast to \cite[Theorem 1.3]{CAI23} where the small initial energy is needed.
\end{remark}

\begin{remark}
Our work extends the significant results of Danchin and Mucha in \cite{DM17,DM23} to two-dimensional bounded convex domains under the slip boundary conditions. It should be noted that the methods in \cite{DM23} seem not to be directly applicable in the presence of boundaries and non-zero initial total momentum. Moreover, compared with \cite[Theorem 1.1 and Corollary 1.1]{DM17}, the density in our theorems is allowed to have large oscillation and vacuum states.
\end{remark}

\begin{remark}
Theorems \ref{t1.1} and \ref{t1.2} are generalizations of our previous work \cite{WWZ} in the half-plane case. However, the task is non-trivial as the prior approach relies heavily on the geometric flatness of the boundary. Looking forward, it is natural to ask whether analogous results
hold for 2D bounded domains with distinct boundary conditions, such as Dirichlet or free boundary conditions. New ideas are required to handle these cases, which will be left for future studies.
\end{remark}

\begin{remark}
We emphasize that the convexity of domain $\Omega$ plays a crucial role in establishing a Desjardins-type logarithmic interpolation inequality \eqref{z2.1}. It seems difficult for us to remove this technical condition.
\end{remark}


\subsection{Strategy of the proof}
We shall now explain the major difficulties and techniques in the proof. The demonstration of Theorem \ref{t1.1} necessitates the construction of global smooth approximate solutions. To this aim, we mainly adopt the local existence theory with strictly positive initial density and the blow-up criterion (see Lemma \ref{l2.1}), and then let the lower bound of the initial density go to zero. The pivotal aspect of our analytical framework involves deriving uniform \textit{a priori} estimates which are independent of the lower bound of density and the bulk viscosity, thus enabling us to study the incompressible limit.

It should be pointed out that the crucial techniques used in \cite{DM23,WWZ} cannot be adopted to the situation treated here since their arguments rely heavily on the facts that domain is without boundary or flat. Moreover, compared with the case of 3D bounded domains \cite{CAI23}, the absence of small initial energy and the bounds of {\it a priori} estimates independent of the bulk viscosity $\lambda$ will give rise to additional difficulties. Consequently, some new observations and ideas are needed to overcome these obstacles.

It was shown in \cite{SZ11} that if $0<T^*<\infty$ is the maximal existence time of strong solutions to \eqref{a1}--\eqref{a3}, then
\begin{equation*}
\limsup\limits_{T\nearrow T^*}\|\rho\|_{L^\infty(0,T;L^\infty)}=\infty,
\end{equation*}
which implies that the key issue is to obtain the time-independent upper bound of the density. One of crucial steps is to establish the time-independent upper bound of $\|\nabla\mathbf{u}\|_{L^2}^2$ and especially $(2\mu+\lambda)\|\divv\mathbf{u}\|_{L^2}^2$ (see Lemma \ref{l3.2}).
For this purpose, a key point is to handle the term $\int\rho \dot{\mathbf{u}}\cdot(\mathbf{u}\cdot\nabla)\mathbf{u}\mathrm{d}\mathbf{x}$ in \eqref{z3.9}, where the direct $L^4$-estimate of $\mathbf{u}$ does not suffice. So we turn to considering
the $L^\infty$-norm of $\mathbf{u}$ in two-dimensional space.
Fortunately, the lack of $H^1\hookrightarrow L^\infty$ can be remedied with the aid of Moser--Trudinger inequality, where $L^\infty$-boundedness is replaced by an exponential integrability. This critical phenomenon stems from the logarithmic divergence of Green's function and is deeply linked to the conformal structure of the sphere (via stereographic projection). We thus deal with $\|\sqrt{\rho}\mathbf{u}\|_{L^4}$ via a Desjardins-type logarithmic interpolation inequality (see Lemma $\ref{log}$), from which one sees that
$\int\rho \dot{\mathbf{u}}\cdot(\mathbf{u}\cdot\nabla)\mathbf{u}\mathrm{d}\mathbf{x}$
can be controlled in terms of $\frac{1}{(2\mu+\lambda)^2}\|P(\rho)-\bar{P}\|_{L^4}^4$ (see  \eqref{z3.10}--\eqref{z3.11}). Hence, we assume the \textit{a priori hypothesis} for the integrability in time of $\frac{1}{(2\mu+\lambda)^2}\|P-\Bar{P}\|_{L^4}^4$ (see \eqref{z3.1}). To complete the proof of the \textit{a priori hypothesis} (that is, one needs to show \eqref{z3.2}), we observe from the momentum equation \eqref{a1}$_2$ that
\begin{equation*}
\divv\mathbf{u}=\frac{-(-\Delta)^{-1}\divv(\rho\dot{\mathbf{u}})+P-\bar{P}}{2\mu+\lambda},
\end{equation*}
which motivates us to build the estimates on the material derivative of the velocity (see Lemmas $\ref{l3.2}$ and $\ref{l3.4}$). In particular, distinct from conventional approaches for this upper bound of the integrability in time in bounded domains, we circumvent the utilization of Bogovskii operator due to the assumptions on the bulk viscosity (see \eqref{z3.27}).

However, since the bulk viscosity emerges in conjunction with the divergence of the velocity, it seems to be difficult to extract information of the effective viscous flux from $\nabla\mathbf{u}$. Indeed, the consideration of deriving the $L^\infty(0,\min\{1,T\};L^2)$-norm for the material derivative of the velocity (see Lemma $\ref{l3.4}$) immediately presents new challenges.
In contrast to prior literature concerning slip boundary problems, we have to apply Hoff's strategy to the equation
\begin{equation*}
\rho\dot{\mathbf{u}}-(2\mu+\lambda)\nabla\divv \mathbf{u}+\mu\nabla^\bot\curl \mathbf{u}+\nabla P=0
\end{equation*}
rather than
\begin{equation*}
  \rho\dot{\mathbf{u}}=\nabla F-\mu\nabla^\bot\curl\mathbf{u},
\end{equation*}
which gives rise to numerous intractable terms in the absence of the effective viscous flux. In view of this, we can devise new approaches to address these issues, as detailed below.

First, it is very hard to obtain the estimates for the term
\begin{equation*}
  (2\mu+\lambda)\sigma\int\dot{u}^j\left(\partial_j\divv\mathbf{u}_t +\divv(\mathbf{u}\partial_j\divv\mathbf{u})\right)\mathrm{d}\mathbf{x}.
\end{equation*}
As a fallback, motivated by \cite{HWZ}, we shall pursue (see \eqref{z3.31})
\begin{equation*}
  (2\mu+\lambda)\int(\divf\mathbf{u})^2\mathrm{d}\mathbf{x}~~~~
\text{rather than}
~~~~(2\mu+\lambda)\int(\divv\dot{\mathbf{u}})^2\mathrm{d}\mathbf{x}.
\end{equation*}
But this will make the estimates more delicate and complicated (see \eqref{z3.31}--\eqref{z3.34}),
particularly concerning the additional boundary integrals (see \eqref{z3.35} and \eqref{z3.38}).

Second, we have to use the Hodge-type decomposition, because only partial information (curl-free part) from $\|\nabla\mathbf{u}\|_{L^p}$ can be controlled in combination
with the bulk viscosity $\lambda$. How to obtain the estimate for the divergence-free part is a core challenge.
An important consideration is to avoid the occurrence of boundary integrals like
\begin{equation*}
  \int_{\partial\Omega} F(\mathcal{P}\mathbf{u})_t\cdot\nabla(\mathcal{P}\mathbf{u})\cdot\mathbf{n}\mathrm{ds} \ \
   \text{and} \ \ \int_{\partial\Omega} F(\mathcal{P}\mathbf{u})\cdot\nabla(\mathcal{P}\mathbf{u})_t\cdot\mathbf{n}\mathrm{ds}
\end{equation*}
without the information of bulk viscosity on the self-convection of $\mathcal{P}\mathbf{u}$. This also represents one of our primary objectives in \eqref{z3.32}, distinct from the half-plane case where integration by parts plays a crucial role. More precisely, if we turn our attention to the viscosity occurred in the effective viscous flux $F$, the Gagliardo--Nirenberg inequality yields that the term $P(\rho)-\bar{P}$ (or $P(\rho)-P(\tilde{\rho})$ in the half-plane case \cite{WWZ}) suppresses the power of $\lambda$ in the sense of
\begin{equation*}
\|F\|_{L^p}\lesssim(2\mu+\lambda)^\frac2p\cdots.
\end{equation*}
For the domain bounded by a closed boundary, the Poincar\'{e} inequality yields uniform-in-$\lambda$ estimates:
\begin{align*}
\|F\|_{L^p}\leq C\|F\|_{H^1}\leq C\|\nabla F\|_{L^2},
\end{align*}
which seems to suggest that the term $P(\rho)-\bar{P}$ in the bounded domains exerts a stronger damping effect on the viscosity coefficients than $P(\rho)-P(\tilde{\rho})$ does in the half-plane. Meanwhile, the higher regularity of weak solutions obtained allows us to work with $\|\nabla^2\mathcal{P}\mathbf{u}\|_{L^2}$ instead of the more problematic $\|\nabla^2\mathbf{u}\|_{L^2}$, resulting in no boundary terms with $\mathcal{P}\mathbf{u}$ in \eqref{z3.32}.

Last but not least, the additional presence of the boundary integral terms is analytically intractable, because density is not defined on the boundary and the trace theorem will fail in some cases (see \eqref{z3.38}). To this end, we address the divergence theorem and achieve cancelation among a subset of boundary terms (see \eqref{z3.35} and \eqref{z3.38}),
circumventing the estimates for some terms such as
\begin{equation*}
  \int_{\partial\Omega} P\dot{\mathbf{u}}\cdot\nabla\mathbf{u}\cdot\mathbf{n}\mathrm{ds}~~\text{and}~~\int_{\partial\Omega}P_t\dot{\mathbf{u}}\cdot\mathbf{n}\mathrm{ds}.
\end{equation*}
Moreover, an observation of
\begin{equation*}
-\int_{\partial\Omega}\bar{P}_t\dot{\mathbf{u}}\cdot\mathbf{n}\mathrm{ds}=
(\gamma-1)\overline{P\divv\mathbf{u}}\int_{\partial\Omega}\dot{\mathbf{u}}\cdot\mathbf{n}\mathrm{ds}
\end{equation*}
leads us to consider (see \eqref{z3.35})
\begin{equation*}
\int_{\partial\Omega}(F-\bar{P})_t\dot{\mathbf{u}}\cdot\mathbf{n}\mathrm{ds}~~~~
\text{rather than}
~~~~\int_{\partial\Omega}F_t\dot{\mathbf{u}}\cdot\mathbf{n}\mathrm{ds}.
\end{equation*}
Furthermore, the application of the divergence theorem and crucial integration by parts enables us to avoid estimating some second-order spatial derivatives (see \eqref{z3.38}).

With the aforementioned difficulties resolved,
we then succeed in deriving the desired estimates on $L^\infty(0,\min\{1,T\};L^2)$-norm (see Lemma $\ref{l3.4}$).
Having these at hand, we can
obtain the time-independent upper bound of the density by applying Lagrangian
coordinates technique used in \cite{DE97} (see Lemma $\ref{l3.5}$).
Consequently, the proof of \textit{a priori hypothesis} is complete once the bulk viscosity is properly large (see {\it Proof of Proposition \ref{p3.1}}).
We point out that the effective viscous flux and Desjardins-type logarithmic interpolation inequality play essential roles in our analysis.

The rest of the paper is organized as follows. In the next section, we recall some known facts and elementary inequalities that will be used later. Section \ref{sec3} is devoted to obtaining {\it a priori} estimates. The proofs of Theorems \ref{t1.1} and \ref{t1.2} are presented in Sections \ref{sec4} and \ref{sec5}, respectively.

\section{Preliminaries}\label{sec2}

In this section, we collect some facts and elementary inequalities that will be used frequently later.

\subsection{Auxiliary results and inequalities}
In this subsection, we review some known lemmas and facts. First of all, we recall the following results concerning the local existence and the possible breakdown of strong solutions to the problem \eqref{a1}--\eqref{a3}, which have been proven in \cite{MN80} and \cite{SZ11}, respectively.
\begin{lemma}\label{l2.1}
Assume that
\begin{equation*}
\rho_0\in H^{2},~~\inf\limits_{\mathbf{x}\in\Omega}\rho_0(\mathbf{x})>0,~~
\mathbf{u}_0\in  H_\omega^2,
\end{equation*}
then there exists a positive constant $T$ such that the initial-boundary value problem \eqref{a1}--\eqref{a3} admits a unique strong solution $(\rho,\mathbf{u})$ satisfying
\begin{equation*}
(\rho,\mathbf{u})\in C([0,T]; H^{2}),\ \ \inf_{\Omega\times[0,T]}\rho(\mathbf{x},t)\geq\frac{1}{2}\inf_{\mathbf{x}\in\Omega}\rho_0(\mathbf{x})>0.
\end{equation*}
Moreover, if $T^*$ is the maximal time of existence, then it holds that
\begin{equation*}
\lim\sup_{T\nearrow T^*}\|\rho\|_{L^\infty(0,T;L^\infty)}=\infty.
\end{equation*}
\end{lemma}

The following Gagliardo--Nirenberg inequality (see \cite[Remark 2.1]{LWZ}) will be used frequently later.
\begin{lemma}\label{GN}
(Gagliardo--Nirenberg inequality, special case).
Assume that $\Omega$ is a bounded Lipschitz domain in $\mathbb{R}^2$.
For $p\in [2, \infty)$, $q\in(1, \infty)$, and $r\in (2, \infty)$, there exist generic constants $C_i>0\ (i\in\{1,2,3,4\})$ which may depend only on $p$, $q$, $r$, and $\Omega$ such that, for $f\in H^1$ and $g\in L^q$ with $\nabla g\in L^{r}$,
\begin{gather*}
\|f\|_{L^p}\le C_1\|f\|_{L^2}^\frac{2}{p}\|\nabla f\|_{L^2}^{1-\frac{2}{p}}+C_2\|f\|_{L^2},\\
\|g\|_{L^\infty}\le C_3\|g\|_{L^q}^\frac{q(r-2)}{2r+q(r-2)}\|\nabla g\|_{L^r}^\frac{2r}{2r+q(r-2)}+C_4\|g\|_{L^2}.
\end{gather*}
Moreover, if $\int_{\Omega}f(\mathbf{x})\mathrm{d}\mathbf{x}=0$ or $(f\cdot n)|_{\partial\Omega}=0$, we can choose $C_2=0$. Similarly, the constant $C_4=0$ provided $\int_{\Omega}g(\mathbf{x})\mathrm{d}\mathbf{x}=0$ or $(g\cdot n)|_{\partial\Omega}=0$.
\end{lemma}

Next, we introduce a generalized Poincar{\'e}'s inequality (see \cite[Lemma 8]{BS2012}) adapted to our setting.
\begin{lemma}\label{PO}
Let $\Omega\subset\mathbb{R}^2$ be a bounded Lipschitz domain. Then, for $1<p<\infty$, there exists a positive constant $C$ depending only on $p$ and $\Omega$ such that
\begin{equation*}
\|f\|_{L^p} \leq  C\|\nabla f\|_{L^p},
\end{equation*}
for each vector field $f\in W^{1,p}(\Omega)$ satisfying either $\int_{\Omega}f(\mathbf{x})\mathrm{d}\mathbf{x}=0$ or $(f\cdot n)|_{\partial\Omega}=0$.
\end{lemma}

In addition, we have the following Desjardins-type logarithmic interpolation inequality, which extends the case of two-dimensional torus $\mathbb{T}^2$ in \cite[Lemma 2]{DE97} (see also \cite[Lemma 1]{D1997}) to the general bounded convex domains.
\begin{lemma}\label{log}
Let $\Omega\subset\mathbb{R}^2$ be a bounded convex domain with smooth boundary. Assume that $0\le\rho\le\hat{\rho}$ and $\mathbf{u}\in H^1$, then it holds that
\begin{equation}\label{z2.1}
\|\sqrt\rho\mathbf{u}\|_{L^4}^2\leq C(\hat{\rho},\Omega)(1+\|\sqrt\rho\mathbf{u}\|_{L^2})\|\nabla\mathbf{u}\|_{L^2}
\sqrt{\ln\left(2+\|\nabla\mathbf{u}\|_{L^2}^2\right)},
\end{equation}
where and in what follows we sometimes use $C(f)$ to emphasize the dependence on $f$.
\end{lemma}
\begin{proof}
Let $\varphi(\mathbf{x})\in L^\infty$ be a nonnegative smooth function with
\begin{equation*}
\int_\Omega\varphi(\mathbf{x})\mathrm{d}\mathbf{x}=1.
\end{equation*}
For positive integer $n$, define $\varphi_n(\mathbf{x})= n^2\varphi(n\mathbf{x})$ and $\mathbf{u}_n(\mathbf{x}) = \mathbf{u} \ast \varphi_n(\mathbf{x})=\int_\Omega \mathbf{u}(\mathbf{x}-\mathbf{y})\varphi_n(\mathbf{y})\mathrm{d}\mathbf{y}$. Then it follows from H\"older's inequality and Lemma $\ref{GN}$ that
\begin{align}
\|\sqrt{\rho}\mathbf{u}\|_{L^4}^2&=\|\rho|\mathbf{u}|^2\|_{L^2}\notag\\
&\leq \|\rho\mathbf{u}\cdot(\mathbf{u}-\mathbf{u}_n)\|_{L^2}
+\|\rho\mathbf{u}\cdot\mathbf{u}_n\|_{L^2}\notag\\
&\leq \|\rho\mathbf{u}\|_{L^4}\|\mathbf{u}-\mathbf{u}_n\|_{L^4}+\|\rho\mathbf{u}\|_{L^2}
\|\mathbf{u}_n\|_{L^\infty}\notag\\
&\leq C\hat{\rho}^\frac{1}{2}\|\mathbf{u}-\mathbf{u}_n\|_{L^2}^{\frac{1}{2}}
\|\nabla(\mathbf{u}-\mathbf{u}_n)\|_{L^2}^{\frac{1}{2}}\|\sqrt{\rho}\mathbf{u}\|_{L^4}+ C\hat{\rho}^\frac{1}{2}\|\mathbf{u}-\mathbf{u}_n\|_{L^2}\|\sqrt{\rho}\mathbf{u}\|_{L^4}+
\hat{\rho}^\frac{1}{2}\|\sqrt{\rho}\mathbf{u}\|_{L^2}
\|\mathbf{u}_n\|_{L^\infty}\notag\\
&\leq C(\Omega)\hat{\rho}^\frac{1}{2}\left(\frac{1}{n^\frac{1}{2}}+\frac{1}{n}\right)\|\nabla\mathbf{u}\|_{L^2}\|\sqrt{\rho}\mathbf{u}\|_{L^4}+
\hat{\rho}^\frac{1}{2}\|\sqrt{\rho}\mathbf{u}\|_{L^2}
\|\mathbf{u}_n\|_{L^\infty}\notag\\
&\leq \frac{1}{2}\|\sqrt{\rho}\mathbf{u}\|_{L^4}^2+\frac{C(\hat{\rho},\Omega)}{n}\|\nabla\mathbf{u}\|_{L^2}^2+
\hat{\rho}^\frac{1}{2}\|\sqrt{\rho}\mathbf{u}\|_{L^2}
\|\mathbf{u}_n\|_{L^\infty}\notag,
\end{align}
where we have used the inequality
\begin{equation*}
\|\mathbf{u}-\mathbf{u}_{n}\|_{L^{2}}=\|\mathbf{u}-\mathbf{u}*\varphi_{n}\|_{L^{2}}\leq \frac{C(\Omega)}{n}\|\nabla\mathbf{u}\|_{L^{2}},
\end{equation*}
which can be deduced from Lemma 1.50\footnote{We point out that the convexity of domain $\Omega$ is used in \cite[Lemma 1.50]{M97}.} and Theorem 1.52 in \cite{M97}. Thus, we have
\begin{equation}\label{z2.2}
\|\sqrt{\rho}\mathbf{u}\|_{L^4}^2\leq \frac{C(\hat{\rho},\Omega)}{n}\|\nabla\mathbf{u}\|_{L^2}^2+
2\hat{\rho}^\frac{1}{2}\|\sqrt{\rho}\mathbf{u}\|_{L^2}
\|\mathbf{u}_n\|_{L^\infty}.
\end{equation}

For $0\le a,b<\infty$, in view of Young's inequality
\begin{equation*}
ab\leq\exp(a^2)-1+Cb\sqrt{\ln(2+b)}, ~~\text{for some constant}~C>0,
\end{equation*}
and the Moser--Trudinger inequality concerning bounded connected $C^{1,\alpha} (\alpha\in(0,1])$ domain in \cite{TR67}
\begin{equation*}
\int_{\Omega}\exp\left(\frac{c|f(\mathbf{x})-\bar{f}|^2}{\|\nabla f\|_{L^2}^2}\right)\mathrm{d}\mathbf{x}\leq C(\Omega), ~~\text{for some constants}~c, C(\Omega)>0,
\end{equation*}
one gets that
\begin{align}
|\mathbf{u}_n(\mathbf{x})|&=\left|\int_{\Omega}\mathbf{u}(\mathbf{x}-\mathbf{\mathbf{y}})\varphi_n(\mathbf{\mathbf{y}})\mathrm{d}\mathbf{y}\right|\notag\\
&\leq\int_{\Omega}\frac{\sqrt{c}|\mathbf{u}(\mathbf{x}-\mathbf{y})-\bar{\mathbf{u}}|}{\|\nabla\mathbf{u}\|_{L^2}}\frac{\|\nabla\mathbf{u}\|_{L^2}}{\sqrt{c}}\varphi_n(\mathbf{y})\mathrm{d}\mathbf{y}
+\int_{\Omega}|\bar{\mathbf{u}}|\varphi_n(\mathbf{\mathbf{y}})\mathrm{d}\mathbf{y}
\notag\\
&\leq\int_{\Omega}\left(\exp\left(\frac{c|\mathbf{u}(\mathbf{x}-\mathbf{y})-\bar{\mathbf{u}}|^2}{\|\nabla\mathbf{u}\|_{L^2}^2}\right)-1\right)\mathrm{d}\mathbf{y}
+C\int_{\Omega}\|\nabla\mathbf{u}\|_{L^2}\varphi_n(\mathbf{y})\sqrt{\ln(2+\|\nabla\mathbf{u}\|_{L^2}\varphi_n(\mathbf{y}))}\mathrm{d}\mathbf{y}+C\notag\\
&\leq C+C\|\nabla\mathbf{u}\|_{L^2}\sqrt{\ln(2+n^2\|\nabla\mathbf{u}\|_{L^2})}\int_{\Omega}\varphi_n(\mathbf{y})\mathrm{d}\mathbf{y}\notag\\
&=C+C\|\nabla\mathbf{u}\|_{L^2}\sqrt{\ln(2+n^2\|\nabla\mathbf{u}\|_{L^2})}\notag,
\end{align}
which combined with \eqref{z2.2} yields that
\begin{align}
\|\sqrt{\rho}\mathbf{u}\|_{L^4}^2&\leq \frac{C}{n}\|\nabla\mathbf{u}\|_{L^2}^2+C\|\sqrt{\rho}\mathbf{u}\|_{L^2}\left(1+\|\nabla\mathbf{u}\|_{L^2}
\sqrt{\ln(2+n^2\|\nabla\mathbf{u}\|_{L^2})}\right)\notag\\
&\leq C\left(1+\|\sqrt{\rho}\mathbf{u}\|_{L^2}\right)\|\nabla\mathbf{u}\|_{L^2}
\left(\frac{\|\nabla\mathbf{u}\|_{L^2}}{n}+\sqrt{\ln(2+n^2\|\nabla\mathbf{u}\|_{L^2})}\right)\notag.
\end{align}
 Thus we can obtain the desired \eqref{z2.1} by choosing the positive integer $n=\left[\left(2+\|\nabla\mathbf{u}\|_{L^2}^2\right)^\frac{1}{2}\right]$.
\end{proof}

It should be emphasized that the statement of Lemma \ref{log} has nothing to do with the system \eqref{a1}. It is simply an interpolation between different functional spaces.

Next, the following div-curl decomposition is given in \cite{WW92,PL96}.
\begin{lemma}\label{Hodge}
Let $1<q<\infty$ and $\Omega$ be a simply connected bounded domain in $\mathbb{R}^2$ with Lipschitz boundary $\partial\Omega$ (e.g., a bounded convex domain). For $\mathbf{u}\in W^{1,q}$ satisfying $(\mathbf{u}\cdot\mathbf{n})|_{\partial\Omega}=0$, there exists a constant $C=C(q, \Omega)>0$ such that
\begin{equation*}
    \|\nabla \mathbf{u}\|_{L^q}\leq C(\|\divv \mathbf{u}\|_{L^q}+\|\curl \mathbf{u}\|_{L^q}).
\end{equation*}
\end{lemma}

\begin{remark}
When it comes to the exterior domain, the conclusion of Lemma $\ref{Hodge}$ is no longer true since the first Betti number does not vanish. In fact, the authors \cite{DF03} proved that there exists a unique classical solution to the problem
\begin{align*}
\begin{cases}
\divv \mathbf{v}=0,~~\mathbf{x}\in\Omega^c,\\
\curl \mathbf{v}=0,~\mathbf{x}\in\Omega^c,\\
\mathbf{v}\cdot \mathbf{n}=0,~~\mathbf{x}\in\partial\Omega,\\
|\mathbf{v}|\rightarrow0~\mathrm{as}~|\mathbf{x}|\rightarrow\infty,\\
\int_{\partial\Omega}\mathbf{v}\mathrm{ds}=1,
\end{cases}
\end{align*}
where $\Omega^c$ is the complementary set of a simply connected bounded domain $\Omega\subset\mathbb{R}^2$.
\end{remark}

Finally, we recall the following commutator estimates in \cite[Lemma 4.3]{F04}, which play an important role in the mollifier arguments.
\begin{lemma}\label{lcom}
Let $\Omega\subset \mathbb{R}^2$ be a domain.
Let $\rho \in L^{p}(\Omega)$ and ${\bf u} \in W^{1,q}(\Omega)$ be given functions such that $1\le p,q<\infty$ and $\frac{1}{p}+\frac{1}{q}\le 1$. For any $\epsilon>0$, then we have
\begin{equation*}
\|\divv[\rho{\bf u}]_\epsilon-\divv\left([\rho]_\epsilon{\bf u}\right)\|_{L^1(K)}\le C(K)\|\rho\|_{L^p(\Omega)}\|{\bf u}\|_{W^{1,q}(\Omega)},
\end{equation*}
and
\begin{equation*}
\divv[\rho{\bf u}]_\epsilon-\divv\left([\rho]_\epsilon{\bf u}\right)\rightarrow 0\ \text{ in}\ \ L^1(K)\ \ \text{as}\ \epsilon\rightarrow0
\end{equation*}
for any compact set $K\subset \Omega$.
\end{lemma}

\subsection{Uniform estimates for $F$, $\curl\mathbf{u}$, $\nabla\mathbf{u}$, and $\dot{\mathbf{u}}$}

We first give the following estimates for the effective viscous flux $F$, $\mathbf{\curl u}$, and $\mathbf{\nabla u}$, where the constant  $C$ is independent of the bulk viscosity $\lambda$.

\begin{lemma}\label{E0}
Let $(\rho,\mathbf{u})$ be a smooth solution of the problem \eqref{a1}--\eqref{a3}. Then, for any $2\leq p<\infty$, there exists a generic positive constant $C$ depending only on $p,\mu$, and $\Omega$ such that
\begin{gather}\label{E1}
\|\nabla F\|_{L^p}+\|\nabla\curl \mathbf{u}\|_{L^p}+\|\nabla^2\mathcal{P}\mathbf{u}\|_{L^p}\leq C\left(\|\rho\dot{\mathbf{u}}\|_{L^p}+\|\nabla \mathbf{u}\|_{L^p}\right),
\\ \label{E2}
\|\curl\mathbf{u}\|_{L^p}+\|\nabla\mathcal{P}\mathbf{u}\|_{L^p}\leq C\Big(\|\rho\dot{\mathbf{u}}\|_{L^2}^{1-\frac{2}{p}}\|\nabla \mathbf{u}\|_{L^2}^\frac{2}{p}+\|\nabla \mathbf{u}\|_{L^2}\Big),
\\ \label{E3}
\|F\|_{L^p}\leq C\left(\|\rho\dot{\mathbf{u}}\|_{L^2}+\|\nabla \mathbf{u}\|_{L^2}\right),
\\ \label{E4}
\|F\|_{L^p}\leq
C\|\rho\dot{\mathbf{u}}\|_{L^2}^{1-\frac{2}{p}}
\Big((2\mu+\lambda)^\frac{2}{p}\|\nabla\mathbf{u}\|_{L^2}^\frac{2}{p}
+\|P-\bar{P}\|_{L^2}^\frac{2}{p}\Big)
+C\left((2\mu+\lambda)^\frac{2}{p}\|\nabla\mathbf{u}\|_{L^2}
+\|P-\bar{P}\|_{L^2}\right),
\\ \label{E5}
\|\nabla \mathbf{u}\|_{L^p}
\leq C\|\rho\dot{\mathbf{u}}\|_{L^2}^{1-\frac{2}{p}}
\|\nabla\mathbf{u}\|_{L^2}^\frac{2}{p}+C\|\nabla \mathbf{u}\|_{L^2}+\frac{C}{2\mu+\lambda}
\Big(\|\rho\dot{\mathbf{u}}\|_{L^2}^{1-\frac{2}{p}}
\|P-\Bar{P}\|_{L^2}^\frac{2}{p}+
\|P-\Bar{P}\|_{L^p}\Big).
\end{gather}
\end{lemma}
\begin{proof}
It follows from $\eqref{a1}_2$ and \eqref{a3} that
\begin{align*}
\begin{cases}
\mu\Delta\curl \mathbf{u}=\curl(\rho{\dot{\mathbf{u}}}),&\quad \mathbf{x}\in\Omega,\\
\mathbf{u}\cdot \mathbf{n}=0,~~\curl \mathbf{u}=0, &\quad \mathbf{x}\in\partial\Omega,
\end{cases}
\end{align*}
which combined with the standard $L^p$-estimate of elliptic system shows that
\begin{equation*}
  \|\nabla\curl \mathbf{u}\|_{L^p}\leq C\left(\|\rho\dot{\mathbf{u}}\|_{L^p}+\|\nabla \mathbf{u}\|_{L^p}\right),
\end{equation*}
and furthermore,
\begin{equation*}
\|\nabla F\|_{L^p}\leq \|\rho\dot{\mathbf{u}}\|_{L^p}+\mu\|\nabla^\bot\curl \mathbf{u}\|_{L^p}\leq C(\|\rho\dot{\mathbf{u}}\|_{L^p}+\|\nabla \mathbf{u}\|_{L^p})
\end{equation*}
due to $\nabla F=\rho{\dot{\mathbf{u}}}+\mu\nabla^\bot\curl\mathbf{u}$.

Next, we derive from Lemma \ref{GN} and Young's inequality that
\begin{align*}
\|\curl \mathbf{u}\|_{L^p}&\le C\|\curl \mathbf{u}\|_{L^2}^\frac{2}{p}\|\nabla\curl \mathbf{u}\|_{L^2}^{1-\frac{2}{p}}+C\|\curl \mathbf{u}\|_{L^2}\leq C\Big(\|\rho\dot{\mathbf{u}}\|_{L^2}^{1-\frac{2}{p}}\|\nabla \mathbf{u}\|_{L^2}^\frac{2}{p}+\|\nabla \mathbf{u}\|_{L^2}\Big).
\end{align*}
Moreover, noting that $\bar{F}=0$, one deduces from Lemma $\ref{GN}$, \eqref{E1}, Poincar\'{e}'s inequality, and Young's inequality that
\begin{align*}
\|F\|_{L^p}\le C\|F\|_{H^1}\leq C\|\nabla F\|_{L^2}
\leq C\left(\|\rho\dot{\mathbf{u}}\|_{L^2}+\|\nabla \mathbf{u}\|_{L^2}\right),
\end{align*}
and
\begin{align*}
\|F\|_{L^p}&\le C\|F\|_{L^2}^\frac{2}{p}\|\nabla F\|_{L^2}^{1-\frac{2}{p}}\le C\left((2\mu+\lambda)\|\divv\mathbf{u}\|_{L^2}+\|P-\Bar{P}\|_{L^2}\right)^\frac{2}{p}
\left(\|\rho\dot{\mathbf{u}}\|_{L^2}+\|\nabla \mathbf{u}\|_{L^2}\right)^{1-\frac{2}{p}}\\
&\leq C\|\rho\dot{\mathbf{u}}\|_{L^2}^{1-\frac{2}{p}}\Big((2\mu+\lambda)^\frac{2}{p}\|\nabla\mathbf{u}\|_{L^2}^\frac{2}{p}
+\|P-\bar{P}\|_{L^2}^\frac{2}{p}\Big)
+C\left((2\mu+\lambda)^\frac{2}{p}\|\nabla\mathbf{u}\|_{L^2}+\|P-\bar{P}\|_{L^2}\right),
\end{align*}
as the desired \eqref{E3} and \eqref{E4}, where in the last inequality we have used the fact $2\mu+\lambda\geq\mu>0$.

Finally, one gets from Lemma $\ref{Hodge}$ that
\begin{align*}
\|\nabla\mathbf{u}\|_{L^p}&\leq C(\|\divv \mathbf{u}\|_{L^p}+\|\curl \mathbf{u}\|_{L^p})\leq \frac{C}{2\mu+\lambda}(\|F\|_{L^p}+\|P-\Bar{P}\|_{L^p})+C\|\curl \mathbf{u}\|_{L^p}\\
&\leq
\frac{C}{(2\mu+\lambda)^{1-\frac{2}{p}}}\Big(\|\rho\dot{\mathbf{u}}\|_{L^2}^{1-\frac{2}{p}}\|\nabla\mathbf{u}\|_{L^2}^\frac{2}{p}
+\|\nabla\mathbf{u}\|_{L^2}\Big)+\frac{C}{2\mu+\lambda}\Big(\|\rho\dot{\mathbf{u}}\|_{L^2}^{1-\frac{2}{p}}\|P-\bar{P}\|_{L^2}^\frac{2}{p}
+\|P-\bar{P}\|_{L^2}\notag\\&\quad
+\|P-\Bar{P}\|_{L^p}\Big)+C\Big(\|\rho\dot{\mathbf{u}}\|_{L^2}^{1-\frac{2}{p}}\|\nabla \mathbf{u}\|_{L^2}^\frac{2}{p}+\|\nabla \mathbf{u}\|_{L^2}\Big)\\
&\leq C\|\rho\dot{\mathbf{u}}\|_{L^2}^{1-\frac{2}{p}}
\|\nabla\mathbf{u}\|_{L^2}^\frac{2}{p}+C\|\nabla \mathbf{u}\|_{L^2}+\frac{C}{2\mu+\lambda}\Big(\|\rho\dot{\mathbf{u}}\|_{L^2}^{1-\frac{2}{p}}
\|P-\Bar{P}\|_{L^2}^\frac{2}{p}+
\|P-\Bar{P}\|_{L^p}\Big),
\end{align*}
where one has used
\begin{equation*}
  \|P-\Bar{P}\|_{L^2}\leq C(p,\Omega)\|P-\Bar{P}\|_{L^p}
\end{equation*}
for any $2\leq p<\infty$.
\end{proof}

For material derivative $\dot{\mathbf{u}}$ of velocity field, we retell the condition \eqref{a3}. Indeed,
$(\mathbf{u}\cdot \mathbf{n})|_{\partial\Omega}=0$ yields
\begin{equation*}
\mathbf{u}=(\mathbf{u}\cdot\mathbf{n}^\bot)\mathbf{n}^\bot, ~~\text{on}~\partial\Omega,
\end{equation*}
where $\mathbf{n}^\bot\triangleq(n^2,-n^1)$ is the unit tangent vector along the boundary.
Then it follows that
\begin{align}\label{z2.9}
  \dot{\mathbf{u}}\cdot\mathbf{n}&=\mathbf{u}\cdot \nabla\mathbf{u}\cdot\mathbf{n}
  =\mathbf{u}\cdot\nabla(\mathbf{u}\cdot\mathbf{n})-\mathbf{u}\cdot \nabla\mathbf{n}\cdot\mathbf{u}
  =-\mathbf{u}\cdot \nabla\mathbf{n}\cdot\mathbf{u}\\
  &=-(\mathbf{u}\cdot\mathbf{n}^\bot)\mathbf{u}\cdot \nabla\mathbf{n}\cdot\mathbf{n}^\bot
  =(\mathbf{u}\cdot\mathbf{n}^\bot)\mathbf{u}\cdot \nabla\mathbf{n}^\bot\cdot\mathbf{n}
  ,~~\text{on}~\partial\Omega.\notag
\end{align}
That is,
\begin{equation*}
[\dot{\mathbf{u}}-(\mathbf{u}\cdot\mathbf{n}^\bot)(\mathbf{u}\cdot \nabla)\mathbf{n}^\bot]\cdot\mathbf{n}=0,~~ \text{on}~ \partial\Omega,
\end{equation*}
which together with Poincar\'{e}'s inequality ensures the following estimates analogous to \cite[Lemma 2.8]{LWZ}.
\begin{lemma}\label{dot}
Let $(\rho,\mathbf{u})$ be a smooth solution to the problem \eqref{a1}--\eqref{a3}. For any $p\in[1,\infty)$,  there is a positive constant $C$ which may depend only on $p$ and $\Omega$ such that
\begin{gather*}
\|\dot{\mathbf{u}}\|_{L^p}\leq C\left(\|\nabla\dot{\mathbf{u}}\|_{L^2}+\|\nabla \mathbf{u}\|_{L^2}^2\right),\\
\|\nabla\dot{\mathbf{u}}\|_{L^2}\leq C\big(\|\divv\dot{\mathbf{u}}\|_{L^2}+\|\curl \dot{\mathbf{u}}\|_{L^2}+\|\nabla \mathbf{u}\|_{L^4}^2\big).
\end{gather*}
\end{lemma}

\section{\textit{A priori} estimates}\label{sec3}
In this section we will obtain some necessary {\it a priori} bounds for the strong solutions guaranteed by Lemma $\ref{l2.1}$ to the problem \eqref{a1}--\eqref{a3}. We should point out that these bounds are independent of the bulk viscosity $\lambda$, the lower bound of $\rho$, the initial regularity, and the time of existence. More precisely, let $T>0$ be fixed and $(\rho, \mathbf{u})$ be the strong solution to \eqref{a1}--\eqref{a3} in $\Omega\times(0, T]$, we can obtain the following key {\it a priori} estimates on $(\rho, \mathbf{u})$.
\begin{proposition}\label{p3.1}
Under the conditions of Theorem $\ref{t1.1}$, if $(\rho, \mathbf{u})$ is a strong solution to the initial-boundary value problem \eqref{a1}--\eqref{a3} satisfying
\begin{align}\label{z3.1}
\sup_{\Omega\times[0,T]}\rho\le2\hat{\rho},\ \ \frac{1}{(2\mu+\lambda)^2}\int_{0}^{T}
\|P-\Bar{P}\|_{L^4}^4\mathrm{d}t\le2,
\end{align}
then one has that
\begin{align}\label{z3.2}
\sup_{\Omega\times[0,T]}\rho\le\frac{7}{4}\hat{\rho},\ \ \frac{1}{(2\mu+\lambda)^2}\int_{0}^{T}
\|P-\Bar{P}\|_{L^4}^4\mathrm{d}t\le1.
\end{align}
\end{proposition}

Before proving Proposition \ref{p3.1}, we first establish some necessary \textit{a priori} estimates, see Lemmas \ref{l3.1}--\ref{l3.4} below. Let us start with the elementary energy estimate of $(\rho, \mathbf{u})$.
\begin{lemma}\label{l3.1}
It holds that
\begin{align}\label{z3.3}
\sup_{0\le t\le T}\int\bigg(\frac{1}{2}\rho |\mathbf{u}|^2+G(\rho)\bigg)\mathrm{d}\mathbf{x}+\int_0^T\left[(2\mu+\lambda)\|\divv \mathbf{u}\|_{L^2}^2+\mu\|\curl \mathbf{u}\|_{L^2}^2\right]\mathrm{d}t\le C_0.
\end{align}
\end{lemma}
\begin{proof}
By virtue of $\eqref{a1}_1$ and the definition of $G(\rho)$ in \eqref{G}, we conclude that
\begin{equation}\label{z3.4}
(G(\rho))_t+\divv (G(\rho)\mathbf{u})+(P-P(\bar{\rho}))\divv \mathbf{u}=0.
\end{equation}
Integrating \eqref{z3.4} over $\Omega$, one gets that
\begin{align}\label{z3.5}
\frac{\mathrm{d}}{\mathrm{d}t}\int G(\rho)\mathrm{d}\mathbf{x}+\int(P-P(\bar{\rho}))\divv \mathbf{u}\mathrm{d}\mathbf{x}=0.
\end{align}

Owing to the vector identity $\Delta \mathbf{u}=\nabla\divv \mathbf{u}-\nabla^\bot\curl \mathbf{u}$, we can rewrite $\eqref{a1}_2$ as
\begin{equation}\label{z3.6}
\rho\mathbf{u}_t+\rho\mathbf{u}\cdot\nabla\mathbf{u}-(2\mu+\lambda)\nabla\divv \mathbf{u}+\mu\nabla^\bot\curl \mathbf{u}+\nabla P=0.
\end{equation}
Multiplying \eqref{z3.6} by $\mathbf{u}$ and integrating over $\Omega$,
and then adding the resultant to \eqref{z3.5}, we arrive at
\begin{align}\label{z3.7}
\frac{\mathrm{d}}{\mathrm{d}t}\int \bigg(\frac{1}{2}\rho|\mathbf{u}|^2+G(\rho)\bigg)\mathrm{d}\mathbf{x}+(2\mu+\lambda)\int(\divv \mathbf{u})^2\mathrm{d}\mathbf{x}+\mu\int(\curl \mathbf{u})^2\mathrm{d}\mathbf{x}=0.
\end{align}
Integrating \eqref{z3.7} with respect to $t$ over $(0, T)$ yields \eqref{z3.3}.
\end{proof}

The next lemma provides a time-independent $L^\infty(0, T;L^2)$ estimate for $\divv \mathbf{u}$ and $\curl \mathbf{u}$.

\begin{lemma}\label{l3.2}
Let \eqref{z3.1} be satisfied, then there exists a positive constant $D_2$ depending only on $\hat{\rho}$, $a$, $\gamma$, $\mu$, and $\Omega$ such that
\begin{align}\label{z3.8}
\sup_{0\le t\le T}\left[(2\mu+\lambda)\|\divv \mathbf{u}\|_{L^2}^2+\mu\|\curl \mathbf{u}\|_{L^2}^2\right]+\int_0^T\|\sqrt{\rho} \dot{\mathbf{u}}\|_{L^2}^2\mathrm{d}t\le (2+M)^{e^{2D_2(1+C_0)^2}}
\end{align}
provided that $\lambda$ satisfies \eqref{lam} with $D\geq D_2$.
\end{lemma}
\begin{proof}
Multiplying \eqref{z3.6} by $\mathbf{u}_t$ and integrating the resultant over $\Omega$, we have that
\begin{align}\label{z3.9}
\frac{1}{2}\frac{\mathrm{d}}{\mathrm{d}t}\int\left[(2\mu+\lambda)(\divv \mathbf{u})^2+\mu(\curl \mathbf{u})^2\right]\mathrm{d}\mathbf{x}+\int\rho |\dot{\mathbf{u}}|^2\mathrm{d}\mathbf{x}=-\int\mathbf{u}_t\cdot\nabla P\mathrm{d}\mathbf{x}+\int\rho \dot{\mathbf{u}}\cdot(\mathbf{u}\cdot\nabla)\mathbf{u}\mathrm{d}\mathbf{x}.
\end{align}
It follows from $\eqref{a1}_1$, \eqref{z1.4}, H\"older's inequality, Cauchy--Schwarz inequality, and Lemma $\ref{E0}$ that
\begin{align}\label{z3.10}
&-\int\mathbf{u}_t\cdot\nabla P\mathrm{d}\mathbf{x}=\int P\divv\mathbf{u}_t\mathrm{d}\mathbf{x}=\frac{\mathrm{d}}{\mathrm{d}t}\int P\divv\mathbf{u} \mathrm{d}\mathbf{x}-\int\divv\mathbf{u}P'(\rho)\rho_t\mathrm{d}\mathbf{x}\notag\\
&=\frac{\mathrm{d}}{\mathrm{d}t}\int P\divv\mathbf{u} \mathrm{d}\mathbf{x}+\int\divv\mathbf{u}P'(\rho)\rho\divv\mathbf{u}
\mathrm{d}\mathbf{x}+\int\divv\mathbf{u}P'(\rho)\mathbf{u}
\cdot\nabla\rho\mathrm{d}\mathbf{x}\notag\\
&=\frac{\mathrm{d}}{\mathrm{d}t}\int P\divv\mathbf{u}\mathrm{d}\mathbf{x}+
\int(\divv\mathbf{u})^2P'(\rho)\rho\mathrm{d}\mathbf{x}+
\int\mathbf{u}\cdot\nabla\big(P-\bar{P}\big)\divv\mathbf{u}\mathrm{d}\mathbf{x}\notag\\
&=\frac{\mathrm{d}}{\mathrm{d}t}\int P\divv\mathbf{u}\mathrm{d}\mathbf{x}+
\int(\divv\mathbf{u})^2\big(P'(\rho)\rho-P+\bar{P}\big)\mathrm{d}\mathbf{x}-
\int (P-\bar{P})\mathbf{u}\cdot\nabla\divv\mathbf{u}\mathrm{d}\mathbf{x}\notag\\
&=\frac{\mathrm{d}}{\mathrm{d}t}\int P\divv\mathbf{u}\mathrm{d}\mathbf{x}+
\int(\divv\mathbf{u})^2\big(P'(\rho)\rho-P+\bar{P}\big)\mathrm{d}\mathbf{x}
-\frac{1}{2\mu+\lambda}
\int (P-\bar{P})\mathbf{u}\cdot\nabla\big(F+P-\bar{P}\big)\mathrm{d}\mathbf{x}\notag\\
&=\frac{\mathrm{d}}{\mathrm{d}t}\int P\divv\mathbf{u}\mathrm{d}\mathbf{x}+
\int(\divv\mathbf{u})^2\big(P'(\rho)\rho-P+\bar{P}\big)\mathrm{d}\mathbf{x}
-\frac{1}{2\mu+\lambda}
\int \big(P-\bar{P}\big)\mathbf{u}\cdot\nabla F\mathrm{d}\mathbf{x}\notag\\
&\quad+\frac{1}{4\mu+2\lambda}
\int\big(P-\bar{P}\big)^2\divv\mathbf{u}\mathrm{d}\mathbf{x}\notag\\
&\leq \frac{\mathrm{d}}{\mathrm{d}t}\int P\divv\mathbf{u} \mathrm{d}\mathbf{x}
+C\|\divv \mathbf{u}\|_{L^2}^2+C\|P-\Bar{P}\|_{L^\infty}\|\mathbf{u}\|_{L^2}\|\nabla F\|_{L^2}
+\frac{C}{(2\mu+\lambda)^2}\|P-\Bar{P}\|_{L^4}^4\notag\\
&\leq \frac{\mathrm{d}}{\mathrm{d}t}\int P\divv\mathbf{u} \mathrm{d}\mathbf{x}
+\frac{1}{4}\|\sqrt{\rho}\dot{\mathbf{u}}\|_{L^2}^2+
C\|\nabla\mathbf{u}\|_{L^{2}}^{2}+\frac{C}{(2\mu+\lambda)^2}\|P-\Bar{P}\|_{L^4}^4.
\end{align}
Due to the logarithmic interpolation inequality \eqref{z2.1}, we obtain that
\begin{align}\label{z3.11}
&\int\rho \dot{\mathbf{u}}\cdot(\mathbf{u}\cdot\nabla)\mathbf{u}\mathrm{d}\mathbf{x}\notag\\
&\leq C\|\sqrt{\rho}\dot{\mathbf{u}}\|_{L^2}\|\sqrt{\rho}\mathbf{u}\|_{L^{4}}\|\nabla\mathbf{u}\|_{L^{4}}\notag\\
&\leq C\|\sqrt{\rho}\dot{\mathbf{u}}\|_{L^2}(1+\|\sqrt\rho\mathbf{u}\|_{L^2})^{\frac{1}{2}}
\|\nabla\mathbf{u}\|_{L^2}^{\frac{1}{2}}\ln^{\frac{1}{4}}\left(2+\|\nabla\mathbf{u}\|_{L^2}^2\right)
\Big(\|\sqrt{\rho}\dot{\mathbf{u}}\|_{L^2}^{\frac{1}{2}}
\|\nabla\mathbf{u}\|_{L^2}^\frac{1}{2}+\|\nabla\mathbf{u}\|_{L^{2}}\notag\\
&\quad+\frac{1}{2\mu+\lambda}\|\sqrt{\rho}\dot{\mathbf{u}}\|_{L^2}^{\frac{1}{2}}
\|P-\Bar{P}\|_{L^2}^\frac{1}{2}
+\frac{1}{2\mu+\lambda}
\|P-\Bar{P}\|_{L^4}\Big)\notag\\
&\leq C\|\sqrt{\rho}\dot{\mathbf{u}}\|_{L^2}\Big(1+C_0^{\frac{1}{2}}\Big)^{\frac{1}{2}}
\|\nabla\mathbf{u}\|_{L^{2}}^{\frac{1}{2}}\ln^{\frac{1}{4}}\left(2+\|\nabla\mathbf{u}\|_{L^2}^2\right)
\Big(\|\sqrt{\rho}\dot{\mathbf{u}}\|_{L^2}^{\frac{1}{2}}
\|\nabla\mathbf{u}\|_{L^2}^\frac{1}{2}+\|\nabla\mathbf{u}\|_{L^2}\notag\\
&\quad+\frac{1}{2\mu+\lambda}\|\sqrt{\rho}\dot{\mathbf{u}}\|_{L^2}^{\frac{1}{2}}\|P-\Bar{P}\|_{L^2}^\frac{1}{2}
+\frac{1}{2\mu+\lambda}\|P-\Bar{P}\|_{L^4}\Big)\notag\\
&\leq \frac{1}{4}\|\sqrt{\rho}\dot{\mathbf{u}}\|_{L^2}^2+C(1+C_0)\|\nabla\mathbf{u}\|_{L^2}^2
\big(1+\|\nabla\mathbf{u}\|_{L^{2}}^2\big)
\ln\big(2+\|\nabla\mathbf{u}\|_{L^2}^2\big)
+\frac{C}{(2\mu+\lambda)^4}\|P-\Bar{P}\|_{L^4}^4.
\end{align}

Thus, inserting \eqref{z3.10} and \eqref{z3.11} into \eqref{z3.9}, one has that
\begin{align}\label{z3.17}
&\frac{1}{2}\frac{\mathrm{d}}{\mathrm{d}t}\int\left[(2\mu+\lambda)(\divv \mathbf{u})^2+\mu(\curl \mathbf{u})^2-2P\divv\mathbf{u}\right]\mathrm{d}\mathbf{x}+\frac{1}{2}\int\rho |\dot{\mathbf{u}}|^2\mathrm{d}\mathbf{x}\notag\\
&\leq C(1+C_0)\|\nabla\mathbf{u}\|_{L^2}^2\big(1+\|\nabla\mathbf{u}\|_{L^{2}}^2\big)
\ln\big(2+\|\nabla\mathbf{u}\|_{L^2}^2\big)
+\frac{C}{(2\mu+\lambda)^2}\|P-\Bar{P}\|_{L^4}^4.
\end{align}
According to \eqref{z3.1}, there exists a positive constant $D_1=D_1(a,\gamma,\hat{\rho},\Omega)$ such that
\begin{align}\label{z3.18}
B_1(t)&=\int\left[(2\mu+\lambda)(\divv \mathbf{u})^2+\mu(\curl \mathbf{u})^2-2P\divv\mathbf{u}\right]\mathrm{d}\mathbf{x}\\
&\thicksim (2\mu+\lambda)\|\divv \mathbf{u}\|_{L^2}^2+\mu\|\curl \mathbf{u}\|_{L^2}^2\notag
\end{align}
provided that $\lambda\ge D_1$. Setting
\begin{align}\label{z3.20}
f_1(t)\triangleq2+B_1(t),~~g_1(t)\triangleq(1+C_0)\|\nabla\mathbf{u}\|_{L^2}^2
+\frac{1}{(2\mu+\lambda)^2}\|P-\Bar{P}\|_{L^4}^4,
\end{align}
we then deduce from \eqref{z3.17} that
\begin{equation*}
    f'_{1}(t)\leq Cg_{1}(t)f_{1}(t)\ln f_{1}(t),
\end{equation*}
which leads to
\begin{equation*}
\big(\ln f_1(t)\big)'\leq Cg_1(t)\ln f_1(t).
\end{equation*}
This along with Gronwall's inequality, \eqref{z3.1}, \eqref{z3.3}, and Lemma $\ref{Hodge}$ implies that there exists a positive constant $D_2=D_2(a,\gamma,\mu,\hat{\rho},\Omega)\geq D_1$ such that
\begin{equation}\label{z3.23}
\sup_{0\leq t\leq T}\big[(2\mu+\lambda)\|\divv \mathbf{u}\|_{L^2}^2+\mu\|\curl \mathbf{u}\|_{L^2}^2\big]\leq(2+M)^{\mathrm{e}^{D_2(1+C_0)^2}}
\end{equation}
provided that $\lambda\ge D_2$.

Finally, integrating \eqref{z3.17} over $(0,T)$ along with \eqref{z3.1}, \eqref{z3.3}, and \eqref{z3.23} shows that
\begin{align*}
\int_0^T\|\sqrt{\rho}{\dot{\mathbf{u}}}\|_{L^2}^2\mathrm{d}t
&\leq C(1+M)+C(1+C_0)(2+M)^{\mathrm{e}^{D_2(1+C_0)^2}}
\ln\bigg\{(2+M)^{\mathrm{e}^{D_2(1+C_0)^2}}\bigg\}C_0\notag\\
&\leq(2+M)^{e^{\frac{3}{2}D_2(1+C_0)^2}}
\end{align*}
provided that $\lambda$ satisfies \eqref{lam} with $D\geq D_2$,
which together with \eqref{z3.23} yields the desired \eqref{z3.8}.
\end{proof}

Next, we give the bound of $\frac{1}{2\mu+\lambda}\int_{0}^{T}\|P-\Bar{P}\|_{L^4}^4\mathrm{d}t$.
\begin{lemma}\label{l3.3}
Let \eqref{z3.1} be satisfied, it holds that
\begin{align}\label{z3.25}
\frac{1}{2\mu+\lambda}\int_{0}^{T}\|P-\Bar{P}\|_{L^4}^4\mathrm{d}t\le (2+M)^{e^{3D_2(1+C_0)^2}}
\end{align}
provided that $\lambda$ satisfies \eqref{lam} with $D\geq 2D_2$.
\end{lemma}
\begin{proof}
It follows from $\eqref{a1}_1$ and $P=a\rho^\gamma$ that
\begin{equation*}
  P_t+\divv(P\mathbf{u})+(\gamma-1)P\divv\mathbf{u}=0,
\end{equation*}
which implies that
\begin{equation}\label{z3.26}
  (P-\Bar{P})_t+\mathbf{u}\cdot\nabla(P-\Bar{P})+\gamma P\divv\mathbf{u}-(\gamma-1)\overline{P\divv\mathbf{u}}=0.
\end{equation}
Noting that
\begin{equation*}
  \overline{P\divv\mathbf{u}}=\fint a\rho^\gamma\divv\mathbf{u}\mathrm{d}\mathbf{x}\leq C(a,\gamma,\hat{\rho},\Omega)\left\|\divv\mathbf{u}\right\|_{L^2},
\end{equation*}
we multiply $\eqref{z3.26}$ by $3(P-\Bar{P})^2$ and integrate the resulting equality over $\Omega$ to obtain that
\begin{align}
&\frac{3\gamma-1}{2\mu+\lambda}\left\|P-\Bar{P}\right\|_{L^4}^4\notag\\
&=-\frac{\mathrm{d}}{\mathrm{d}t}\int(P-\Bar{P})^3\mathrm{d}\mathbf{x}
-\frac{3\gamma-1}{2\mu+\lambda}\int(P-\Bar{P})^3F\mathrm{d}\mathbf{x}
+3(\gamma-1)\overline{P\divv\mathbf{u}}\int(P-\Bar{P})^2\mathrm{d}\mathbf{x}\notag\\
&\quad-3\gamma \Bar{P}\int(P-\Bar{P})^2\divv\mathbf{u}\mathrm{d}\mathbf{x}\notag\\
&\leq-\frac{\mathrm{d}}{\mathrm{d}t}\int(P-\Bar{P})^3\mathrm{d}\mathbf{x}+
\frac{3\gamma-1}{2(2\mu+\lambda)}\left\|P-\Bar{P}\right\|_{L^4}^4+\frac{C}{2\mu+\lambda}\left\|F\right\|_{L^4}^4
+C(2\mu+\lambda)\left\|\divv\mathbf{u}\right\|_{L^2}^2\notag.
\end{align}
Integrating the above inequality over $(0,T)$, one deduces from \eqref{z3.3}, \eqref{z3.8}, and Lemma \ref{E0} that
\begin{align}\label{z3.27}
&\frac{1}{2\mu+\lambda}\int_{0}^{T}\|P-\Bar{P}\|_{L^{4}}^{4}\mathrm{d}t\notag\\
&\leq C\sup_{0\leq t\leq T}\left\|P-\Bar{P}\right\|_{L^3}^3+\frac{C}{2\mu+\lambda}\int_0^T\|F\|_{L^2}^2\|\nabla F\|_{L^2}^2\mathrm{d}t+CC_0\notag\\
&\leq C(1+C_0)+\frac{C}{2\mu+\lambda}\int_0^T\left[(2\mu+\lambda)^2
\|\divv\mathbf{u}\|_{L^2}^2+\|P-\Bar{P}\|_{L^2}^2\right]\left(\|\sqrt{\rho}\dot{\mathbf{u}}\|_{L^2}^2+\|\nabla \mathbf{u}\|_{L^2}^2\right)\mathrm{d}t\notag\\
&\leq(2+M)^{\mathrm{e}^{3D_{2}(1+C_{0})^{2}}},
\end{align}
as the desired \eqref{z3.25}.
\end{proof}

Next, motivated by \cite{Hoff95,Hoff95*,HWZ}, we have the following time-weighted estimate on  $\|\sqrt{\rho}{\dot{\mathbf{u}}}\|_{L^2}^2$.
\begin{lemma}\label{l3.4}
Let \eqref{z3.1} be satisfied, it holds that
\begin{align}\label{z3.28}
\sup_{0\leq t\leq T}\big(\sigma\|\sqrt{\rho}{\dot{\mathbf{u}}}\|_{L^2}^2\big)
+\int_0^T\big[(2\mu+\lambda)\sigma\|\divf \mathbf{u}\|_{L^2}^2+\mu\sigma\|\curl \dot{\mathbf{u}}\|_{L^2}^2\big]\mathrm{d}t
\leq\exp\bigg\{(2+M)^{\mathrm{e}^{4D_{2}(1+C_{0})^{2}}}\bigg\}
\end{align}
provided that $\lambda$ satisfies \eqref{lam} with $D\geq 3D_2$.
\end{lemma}
\begin{proof}
Operating $\sigma\dot{u}^j[\partial/\partial t+\divv({\mathbf{u}}\cdot)]$ on $\eqref{z3.6}^j$, summing all the equalities with respect to $j$, and integrating the resultant over $\Omega$, we obtain from
$\eqref{a1}_1$ and \eqref{a3} that
\begin{align}\label{z3.29}
&\frac{1}{2}\frac{\mathrm{d}}{\mathrm{d}t}\int\sigma\rho|\dot{\mathbf{u}}|^2\mathrm{d}\mathbf{x}
-\frac{\sigma'}{2}\int\rho|\dot{\mathbf{u}}|^2\mathrm{d}\mathbf{x}\notag\\
&=-\sigma\int\dot{u}^j[\partial_j P_{t}+\divv(\mathbf{u}\partial_{j}P)]\mathrm{d}\mathbf{x}
-\mu\sigma\int\dot{u}^j\big[(\nabla^\bot\curl \mathbf{u}_{t})^j+\divv\big(\mathbf{u}(\nabla^\bot\curl \mathbf{u})^{j}\big)\big]\mathrm{d}\mathbf{x}\notag\\
&\quad+(2\mu+\lambda)\sigma\int\dot{u}^j[\partial_j\divv\mathbf{u}_t +\divv(\mathbf{u}\partial_j\divv\mathbf{u})]\mathrm{d}\mathbf{x}\triangleq \sum_{i=1}^{3}J_i.
\end{align}
Taking advantage of $\eqref{a1}_1$, integration by parts, and Cauchy--Schwarz inequality, one gets that
\begin{align}\label{z3.30}
J_{1}
&=-\sigma\int_{\partial\Omega}P_{t}\dot{\mathbf{u}}\cdot \mathbf{n}\mathrm{ds}+\sigma\int P_{t}\divv\dot{\mathbf{u}}\mathrm{d}\mathbf{x}-\sigma\int\dot{\mathbf{u}}\cdot\nabla\divv(P\mathbf{u})\mathrm{d}\mathbf{x}
+\sigma\int\dot{u}^j\divv(P\partial_{j}\mathbf{u})\mathrm{d}\mathbf{x}\notag\\
&=-\sigma\int_{\partial\Omega}\big(P_{t}+\divv(P\mathbf{u})\big)\dot{\mathbf{u}}\cdot \mathbf{n}\mathrm{ds}
+\sigma\int \big(P_{t}+\divv(P\mathbf{u})\big)\divv\dot{\mathbf{u}}\mathrm{d}\mathbf{x}
+\sigma\int\dot{\mathbf{u}}\cdot\nabla\mathbf{u}\cdot\nabla P\mathrm{d}\mathbf{x}
\notag\\&\quad+\sigma\int P\dot{\mathbf{u}}\cdot\nabla\divv\mathbf{u}\mathrm{d}\mathbf{x}\notag\\
&=-\sigma\int_{\partial\Omega}\big(P_{t}+\divv(P\mathbf{u})\big)\dot{\mathbf{u}}\cdot \mathbf{n}\mathrm{ds}+\sigma\int_{\partial\Omega} P\dot{\mathbf{u}}\cdot\nabla\mathbf{u}\cdot\mathbf{n}\mathrm{ds}
\notag\\ & \quad-\sigma(\gamma-1)\int P\divv\mathbf{u}\divv\dot{\mathbf{u}}\mathrm{d}\mathbf{x}
-\sigma\int P\nabla\dot{\mathbf{u}}:\nabla\mathbf{u} \mathrm{d}\mathbf{x}
\notag\\
&\triangleq\sum_{i=1}^2\mathcal{B}_i-\sigma(\gamma-1)\int P\divv\mathbf{u}\divv\dot{\mathbf{u}}\mathrm{d}\mathbf{x}
-\sigma\int P\nabla\dot{\mathbf{u}}:\nabla\mathbf{u} \mathrm{d}\mathbf{x}\notag\\
&\leq \sum_{i=1}^2\mathcal{B}_i+\frac{\tilde{C}\sigma}{16}\|\nabla\dot{\mathbf{u}}\|_{L^2}^2
+C\sigma\|\nabla\mathbf{u}\|_{L^2}^2,
\end{align}
where the positive constant $\tilde{C}=\tilde{C}(\mu,\Omega)$ will be determined later.
Similarly, one has that
\begin{align}
J_2&=-\mu\sigma\int\dot{\mathbf{u}}\cdot\nabla^\bot\curl \mathbf{u}_{t}\mathrm{d}\mathbf{x}
-\mu\sigma\int\dot{u}^j\divv(\mathbf{u}\nabla^\bot_j\curl \mathbf{u})\mathrm{d}\mathbf{x}\notag\\
&=-\mu\sigma\int_{\partial\Omega}\curl\mathbf{u}_t\dot{\mathbf{u}}\cdot\mathbf{n}^\bot\mathrm{ds}
-\mu\sigma\int\curl\dot{\mathbf{u}}\curl\mathbf{u}_t\mathrm{d}\mathbf{x}
-\mu\sigma\int\dot{u}^j\big[\nabla^\bot_j\divv(\mathbf{u}\curl \mathbf{u})-\divv(\nabla^\bot_j\mathbf{u}\curl \mathbf{u})\big]\mathrm{d}\mathbf{x}\notag\\
&=-\mu\sigma\int_{\partial\Omega}\curl\mathbf{u}_t\dot{\mathbf{u}}\cdot\mathbf{n}^\bot\mathrm{ds}
-\mu\sigma\int(\curl\dot{\mathbf{u}})^2\mathrm{d}\mathbf{x}
+\mu\sigma\int\curl\dot{\mathbf{u}}\curl(\mathbf{u}\cdot\nabla\mathbf{u})\mathrm{d}\mathbf{x}
\notag\\
&\quad-\mu\sigma\int_{\partial\Omega}\divv(\mathbf{u}\curl\mathbf{u})
\dot{\mathbf{u}}\cdot\mathbf{n}^\bot\mathrm{ds}
-\mu\sigma\int\curl\dot{\mathbf{u}}\divv(\mathbf{u}\curl\mathbf{u})
\mathrm{d}\mathbf{x}
\notag\\
&\quad+\mu\sigma\int_{\partial\Omega}(\curl\mathbf{u})\dot{u}^j\nabla^\bot_j\mathbf{u}\cdot\mathbf{n}\mathrm{ds}
-\mu\sigma\int\nabla\dot{u}^j\cdot\nabla^\bot_j\mathbf{u}\curl\mathbf{u}\mathrm{d}\mathbf{x}\notag\\
&=-\mu\sigma\int(\curl\dot{\mathbf{u}})^2\mathrm{d}\mathbf{x}
-\mu\sigma\int\nabla\dot{u}^j\cdot\nabla^\bot_j\mathbf{u}\curl\mathbf{u}\mathrm{d}\mathbf{x}\notag\\
&\leq-\mu\sigma\|\curl\dot{\mathbf{u}}\|_{L^2}^2+\frac{\tilde{C}\sigma}{16}\|\nabla\dot{\mathbf{u}}\|_{L^2}^2+C\sigma\|\nabla\mathbf{u}\|_{L^4}^4\notag\\
&\leq-\mu\sigma\|\curl\dot{\mathbf{u}}\|_{L^2}^2
+\frac{\tilde{C}\sigma}{16}\|\nabla\dot{\mathbf{u}}\|_{L^2}^2
+C\sigma\big(1+\|\nabla\mathbf{u}\|_{L^2}^2\big)
\big(\|\sqrt{\rho}\dot{\mathbf{u}}\|_{L^2}^2+\|\nabla\mathbf{u}\|_{L^2}^2\big)
\notag\\
&\quad+\frac{C}{(2\mu+\lambda)^4}\|P-\Bar{P}\|_{L^4}^4,
\end{align}
where in the fourth equality we have used $\curl(\mathbf{u}\cdot\nabla\mathbf{u})=\divv(\mathbf{u}\curl\mathbf{u})$.
Turning to the bound of $J_3$, we observe that
\begin{align}\label{z3.31}
J_3
&=(2\mu+\lambda)\sigma\int\dot{u}^j[\partial_j\divv\mathbf{u}_t+ \partial_j\divv(\mathbf{u}\divv\mathbf{u})-\divv(\partial_j\mathbf{u}\divv\mathbf{u})]\mathrm{d}\mathbf{x}\notag\\
&=(2\mu+\lambda)\sigma\int_{\partial\Omega}[\divv\mathbf{u}_t+\divv(\mathbf{u}\divv\mathbf{u})]\dot{\mathbf{u}}\cdot\mathbf{n}\mathrm{ds}
-(2\mu+\lambda)\sigma\int\divv\dot{\mathbf{u}}[\divv\mathbf{u}_t+\divv(\mathbf{u}\divv\mathbf{u})]\mathrm{d}\mathbf{x}\notag\\
&\quad-(2\mu+\lambda)\sigma\int\dot{u}^j\divv(\partial_j\mathbf{u}\divv\mathbf{u})\mathrm{d}\mathbf{x}\notag\\
&=(2\mu+\lambda)\sigma\int_{\partial\Omega}[\divv\mathbf{u}_t+\divv(\mathbf{u}\divv\mathbf{u})]\dot{\mathbf{u}}\cdot\mathbf{n}\mathrm{ds}
-(2\mu+\lambda)\sigma\int\dot{u}^j\divv(\partial_j\mathbf{u}\divv\mathbf{u})\mathrm{d}\mathbf{x}\notag\\
&\quad-(2\mu+\lambda)\sigma\int(\divv\mathbf{u}_t+\mathbf{u}\cdot\nabla\divv\mathbf{u}+\nabla\mathbf{u}:\nabla\mathbf{u})
[\divv\mathbf{u}_t+\mathbf{u}\cdot\nabla\divv\mathbf{u}+(\divv\mathbf{u})^2]\mathrm{d}\mathbf{x}\notag\\
&=(2\mu+\lambda)\sigma\int_{\partial\Omega}\divv\mathbf{u}_t\dot{\mathbf{u}}\cdot\mathbf{n}\mathrm{ds}+
(2\mu+\lambda)\sigma\int_{\partial\Omega}\divv(\mathbf{u}\divv\mathbf{u})\dot{\mathbf{u}}\cdot\mathbf{n}\mathrm{ds}
-(2\mu+\lambda)\sigma\int_{\partial\Omega}(\divv\mathbf{u})\dot{\mathbf{u}}\cdot\nabla\mathbf{u}\cdot\mathbf{n}\mathrm{ds}\notag\\
&\quad-(2\mu+\lambda)\sigma\int\big[(\divf\mathbf{u})^2
+\divf\mathbf{u}(\divv\mathbf{u})^2+\divf\mathbf{u}\nabla\mathbf{u}:\nabla\mathbf{u}
+(\divv\mathbf{u})^2\nabla\mathbf{u}:\nabla\mathbf{u}
-\partial_j\mathbf{u}\cdot\nabla\dot{u}^j\divv\mathbf{u}\big]\mathrm{d}\mathbf{x}\notag\\
&\triangleq \sum_{i=3}^5\mathcal{B}_{i}-(2\mu+\lambda)\sigma\|\divf\mathbf{u}\|_{L^2}^2+\sum_{i=1}^4J_{3i},
\end{align}
where we define
\begin{equation*}
\divf\mathbf{u}\triangleq\divv\mathbf{u}_t+\mathbf{u}\cdot\nabla\divv\mathbf{u}.
\end{equation*}

Next, we shall control each $J_{3i}\ (i\in\{1,2,3,4\})$. Noting that $0\leq\sigma, \sigma'\leq1$ for $t>0$, then one infers from \eqref{z3.1}, H\"older's inequality, Lemma $\ref{GN}$, and Lemma $\ref{E0}$ that
\begin{align}
J_{31}&=-\frac{\sigma}{2\mu+\lambda}\int\divf\mathbf{u}(F+P-\Bar{P})^2\mathrm{d}\mathbf{x}\notag\\
&\leq\frac{C\sigma}{2\mu+\lambda}\|\divf\mathbf{u}\|_{L^2}(\|F\|_{L^4}^2+\|P-\Bar{P}\|_{L^4}^2)\notag\\
&\leq\frac{C\sigma}{2\mu+\lambda}\|\divf\mathbf{u}\|_{L^2}(\|F\|_{L^2}\|\nabla F\|_{L^2}+\|P-\Bar{P}\|_{L^4}^2)\notag\\
&\leq\frac{C\sigma}{2\mu+\lambda}\|\divf\mathbf{u}\|_{L^2}\left[\big((2\mu+\lambda)
\|\divv\mathbf{u}\|_{L^2}+\|P-\Bar{P}\|_{L^2}\big)\left(\|\sqrt{\rho}\dot{\mathbf{u}}\|_{L^2}+\|\nabla \mathbf{u}\|_{L^2}\right)+\|P-\Bar{P}\|_{L^4}^2\right]\notag\\
&\leq\frac{\sigma(2\mu+\lambda)}{16}\|\divf\mathbf{u}\|_{L^2}^2
+C\sigma\big(1+\|\nabla\mathbf{u}\|_{L^2}^2\big)
\big(\|\sqrt{\rho}\dot{\mathbf{u}}\|_{L^2}^2+\|\nabla\mathbf{u}\|_{L^2}^2\big)
+\frac{C}{(2\mu+\lambda)^3}\|P-\Bar{P}\|_{L^4}^4.
\end{align}
Using the Hodge-type decomposition, Lemma $\ref{E0}$, and H\"older's inequality, one gets that
\begin{align}\label{x1}
J_{32}
&=-(2\mu+\lambda)\sigma\int\divf\mathbf{u}(\mathcal{Q}\mathbf{u}+\mathcal{P}\mathbf{u})_j^i(\mathcal{Q}\mathbf{u}
+\mathcal{P}\mathbf{u})_i^j\mathrm{d}\mathbf{x}\notag\\
&=-(2\mu+\lambda)\sigma\int\divf\mathbf{u}\big[(\mathcal{Q}\mathbf{u})_j^i(\mathcal{Q}\mathbf{u})_i^j+
(\mathcal{Q}\mathbf{u})_j^i(\mathcal{P}\mathbf{u})_i^j+(\mathcal{P}\mathbf{u})_j^i(\mathcal{Q}\mathbf{u})_i^j
+(\mathcal{P}\mathbf{u})_j^i(\mathcal{P}\mathbf{u})_i^j\big]\mathrm{d}\mathbf{x}\notag\\
&\leq-(2\mu+\lambda)\sigma\int\divf\mathbf{u}(\mathcal{P}\mathbf{u})_j^i(\mathcal{P}\mathbf{u})_i^j\mathrm{d}\mathbf{x}
+\frac{\sigma(2\mu+\lambda)}{16}\|\divf\mathbf{u}\|_{L^2}^2+C(2\mu+\lambda)\sigma\int|\nabla\mathbf{u}|^2
|\nabla\mathcal{Q}\mathbf{u}|^2\mathrm{d}\mathbf{x}\notag\\
&\leq-(2\mu+\lambda)\sigma\int\divf\mathbf{u}(\mathcal{P}\mathbf{u})_j^i(\mathcal{P}\mathbf{u})_i^j\mathrm{d}\mathbf{x}
+\frac{\sigma(2\mu+\lambda)}{16}\|\divf\mathbf{u}\|_{L^2}^2+C\sigma\|\nabla\mathbf{u}\|_{L^4}^4
+C(2\mu+\lambda)^2\sigma\|\divv\mathbf{u}\|_{L^4}^4\notag\\
&=-(2\mu+\lambda)\sigma\int\divf\mathbf{u}(\mathcal{P}\mathbf{u})_j^i(\mathcal{P}\mathbf{u})_i^j\mathrm{d}\mathbf{x}
+\frac{\sigma(2\mu+\lambda)}{16}\|\divf\mathbf{u}\|_{L^2}^2+C\sigma\|\nabla\mathbf{u}\|_{L^4}^4
+\frac{C\sigma}{(2\mu+\lambda)^2}\|F+P-\Bar{P}\|_{L^4}^4\notag\\
&\leq -\sigma\int F_t(\mathcal{P}\mathbf{u})_j^i(\mathcal{P}\mathbf{u})_i^j\mathrm{d}\mathbf{x}
+C\sigma\big(1+\|\nabla\mathbf{u}\|_{L^2}^2\big)
\big(1+\|\sqrt{\rho}\dot{\mathbf{u}}\|_{L^2}^2+\|\nabla\mathbf{u}\|_{L^2}^2\big)
\big(\|\sqrt{\rho}\dot{\mathbf{u}}\|_{L^2}^2+\|\nabla\mathbf{u}\|_{L^2}^2\big)\notag\\
&\quad
+\frac{\sigma(2\mu+\lambda)}{16}\|\divf\mathbf{u}\|_{L^2}^2
+\frac{C}{(2\mu+\lambda)^2}\|P-\Bar{P}\|_{L^4}^4,
\end{align}
where in the last inequality we have employed
\begin{align*}
&-(2\mu+\lambda)\sigma\int(\divv\mathbf{u}_t+\mathbf{u}\cdot\nabla\divv\mathbf{u})(\mathcal{P}\mathbf{u})_j^i(\mathcal{P}\mathbf{u})_i^j\mathrm{d}\mathbf{x}\notag\\
&=-\sigma\int\big[(P-\bar{P})_t+\mathbf{u}\cdot\nabla(P-\bar{P})+\mathbf{u}\cdot\nabla F\big](\mathcal{P}\mathbf{u})_j^i(\mathcal{P}\mathbf{u})_i^j\mathrm{d}\mathbf{x}
-\sigma\int F_t(\mathcal{P}\mathbf{u})_j^i(\mathcal{P}\mathbf{u})_i^j\mathrm{d}\mathbf{x}\notag\\
&=\sigma\int\big[\gamma P\divv\mathbf{u}-(\gamma-1)\overline{P\divv\mathbf{u}}-\mathbf{u}\cdot\nabla F
\big](\mathcal{P}\mathbf{u})_j^i(\mathcal{P}\mathbf{u})_i^j\mathrm{d}\mathbf{x}
-\sigma\int F_t(\mathcal{P}\mathbf{u})_j^i(\mathcal{P}\mathbf{u})_i^j\mathrm{d}\mathbf{x}\notag\\
&\leq -\sigma\int F_t(\mathcal{P}\mathbf{u})_j^i(\mathcal{P}\mathbf{u})_i^j\mathrm{d}\mathbf{x}
+C\sigma\|\nabla\mathcal{P}\mathbf{u}\|_{L^4}^2\|\nabla\mathbf{u}\|_{L^2}
+C\sigma\|\nabla\mathbf{u}\|_{L^2}^3
+C\sigma\|\mathbf{u}\|_{L^4}\|\nabla\mathcal{P}\mathbf{u}\|_{L^8}^2\|\nabla F\|_{L^2}\notag\\
&\leq -\sigma\int F_t(\mathcal{P}\mathbf{u})_j^i(\mathcal{P}\mathbf{u})_i^j\mathrm{d}\mathbf{x}
+C\sigma(1+\|\nabla\mathbf{u}\|_{L^2}^2)(1+\|\sqrt{\rho}\dot{\mathbf{u}}\|_{L^2}^2+\|\nabla\mathbf{u}\|_{L^2}^2)
(\|\sqrt{\rho}\dot{\mathbf{u}}\|_{L^2}^2+\|\nabla\mathbf{u}\|_{L^2}^2)
\end{align*}
owing to \eqref{z1.4} and \eqref{z3.26}. Indeed, due to Poincar\'{e}'s inequality, Lemma $\ref{GN}$, and Lemma $\ref{E0}$, the term involving $F_t$ in \eqref{x1} can be handled as follows
\begin{align}\label{z3.32}
&-\sigma\int F_t
(\mathcal{P}\mathbf{u})_{j}^{i}(\mathcal{P}\mathbf{u})_{i}^{j}\mathrm{d}\mathbf{x}\notag\\
&=-\frac{\mathrm{d}}{\mathrm{d}t}\int\sigma F
(\mathcal{P}\mathbf{u})_{j}^{i}(\mathcal{P}\mathbf{u})_{i}^{j}\mathrm{d}\mathbf{x}
+\sigma'\int F
(\mathcal{P}\mathbf{u})_{j}^{i}(\mathcal{P}\mathbf{u})_{i}^{j}\mathrm{d}\mathbf{x}
+\sigma\int F
(\mathcal{P}\mathbf{u})_{jt}^{i}(\mathcal{P}\mathbf{u})_{i}^{j}\mathrm{d}\mathbf{x}
+\sigma\int F
(\mathcal{P}\mathbf{u})_{j}^{i}(\mathcal{P}\mathbf{u})_{it}^{j}\mathrm{d}\mathbf{x}\notag\\
&=-\frac{\mathrm{d}}{\mathrm{d}t}\int\sigma F
(\mathcal{P}\mathbf{u})_{j}^{i}(\mathcal{P}\mathbf{u})_{i}^{j}\mathrm{d}\mathbf{x}
+\sigma'\int F
(\mathcal{P}\mathbf{u})_{j}^{i}(\mathcal{P}\mathbf{u})_{i}^{j}\mathrm{d}\mathbf{x}
+\sigma\int F
\big(\mathcal{P}(\dot{\mathbf{u}}-\mathbf{u}\cdot\nabla\mathbf{u})\big)_{j}^{i}(\mathcal{P}\mathbf{u})_{i}^{j}\mathrm{d}\mathbf{x}\notag\\
&\quad+\sigma\int F
(\mathcal{P}\mathbf{u})_{j}^{i}\big(\mathcal{P}(\dot{\mathbf{u}}-\mathbf{u}\cdot\nabla\mathbf{u})\big)_{i}^{j}\mathrm{d}\mathbf{x}\notag\\
&\leq C\|\nabla F\|_{L^2}\|\nabla\mathcal{P}\mathbf{u}\|_{L^4}^2
+C\sigma\|\nabla F\|_{L^2}\|\nabla\mathcal{P}\mathbf{u}\|_{L^4}
\big(\|\nabla\dot{\mathbf{u}}\|_{L^2}
+\|\nabla\mathbf{u}\|_{L^4}^2
+\|\mathbf{u}\|_{L^\infty}\|\nabla^2\mathcal{P}\mathbf{u}\|_{L^2}
\big)\notag\\
&\quad
-\frac{\mathrm{d}}{\mathrm{d}t}\int\sigma F
(\mathcal{P}\mathbf{u})_{j}^{i}(\mathcal{P}\mathbf{u})_{i}^{j}\mathrm{d}\mathbf{x}\notag\\
&\leq-\frac{\mathrm{d}}{\mathrm{d}t}\int\sigma F
(\mathcal{P}\mathbf{u})_{j}^{i}(\mathcal{P}\mathbf{u})_{i}^{j}\mathrm{d}\mathbf{x}
+C\big(1+\|\nabla\mathbf{u}\|_{L^{2}}^{2}\big)
\big(1+\sigma\|\sqrt{\rho}\dot{\mathbf{u}}\|_{L^{2}}^{2}+
\|\nabla\mathbf{u}\|_{L^{2}}^{2}\big)
\big(\|\sqrt{\rho}\dot{\mathbf{u}}\|_{L^{2}}^{2}+\|\nabla\mathbf{u}\|_{L^{2}}^{2}\big)
\notag\\&\quad+\frac{\tilde{C}\sigma}{16}\|\nabla \dot{\mathbf{u}}\|_{L^2}^2
+\frac{C}{(2\mu+\lambda)^4}\|P-\Bar{P}\|_{L^4}^4.
\end{align}
Using similar arguments, we have that
\begin{align}\label{z3.33}
J_{33}&=-(2\mu+\lambda)\sigma\int(\divv\mathbf{u})^2\nabla\mathbf{u}:\nabla\mathbf{u}\mathrm{d}\mathbf{x}\notag\\
&\leq C(2\mu+\lambda)\sigma\int(\divv\mathbf{u})^2\big(|\nabla\mathcal{P}\mathbf{u}|^2+|\nabla\mathcal{Q}\mathbf{u}|^2\big)\mathrm{d}\mathbf{x}\notag\\
&\leq C(2\mu+\lambda)^2\sigma\|\divv\mathbf{u}\|_{L^4}^4+C\sigma\|\nabla\mathcal{P}\mathbf{u}\|_{L^4}^4\notag\\
&\leq\frac{C\sigma}{(2\mu+\lambda)^2}\|F+P-\Bar{P}\|_{L^4}^4+C\sigma(\|\sqrt{\rho}\dot{\mathbf{u}}\|_{L^{2}}^{2}\|\nabla\mathbf{u}\|_{L^{2}}^{2}
+\|\nabla\mathbf{u}\|_{L^{2}}^{4})\notag\\
&\leq C\sigma (1+\|\sqrt{\rho}\dot{\mathbf{u}}\|_{L^2}^2+\|\nabla\mathbf{u}\|_{L^2}^2)(\|\sqrt{\rho}\dot{\mathbf{u}}\|_{L^2}^2+\|\nabla\mathbf{u}\|_{L^2}^2)
+\frac{C}{(2\mu+\lambda)^2}\|P-\Bar{P}\|_{L^4}^4,\\
J_{34}&=(2\mu+\lambda)\sigma\int\partial_j\mathbf{u}\cdot\nabla\dot{u}^j\divv\mathbf{u}\mathrm{d}\mathbf{x}\notag\\
&\leq C\sigma\left(\|\nabla\mathbf{u}\|_{L^4}\|\nabla\dot{\mathbf{u}}\|_{L^2}\|F\|_{L^4}+\|\nabla\mathbf{u}\|_{L^2}\|\nabla\dot{\mathbf{u}}\|_{L^2}\|P-\Bar{P}\|_{L^\infty}\right)\notag\\
&\leq C\sigma\left(\|\nabla\mathbf{u}\|_{L^4}\|\nabla\dot{\mathbf{u}}\|_{L^2}\|\nabla F\|_{L^2}+\|\nabla\mathbf{u}\|_{L^2}\|\nabla\dot{\mathbf{u}}\|_{L^2}\right)\notag\\
&\leq \frac{\tilde{C}\sigma}{16}\|\nabla \dot{\mathbf{u}}\|_{L^2}^2
+C\sigma (1+\|\sqrt{\rho}\dot{\mathbf{u}}\|_{L^2}^2+\|\nabla\mathbf{u}\|_{L^2}^2)(\|\sqrt{\rho}\dot{\mathbf{u}}\|_{L^2}^2+\|\nabla\mathbf{u}\|_{L^2}^2)
+\frac{C}{(2\mu+\lambda)^2}\|P-\Bar{P}\|_{L^4}^4.\label{z3.34}
\end{align}

It remains to estimate the boundary terms $\mathcal{B}_{i}\ (i\in\{1,2,3,4,5\})$. We infer from \eqref{a3}, \eqref{z2.9}, \eqref{z3.26}, and the trace theorem that
\begin{align}\label{z3.35}
&\mathcal{B}_1+\mathcal{B}_3+\mathcal{B}_4
=\sigma\int_{\partial\Omega}[-P_{t}-\divv(P\mathbf{u})+(2\mu+\lambda)\divv\mathbf{u}_t+
(2\mu+\lambda)\divv(\mathbf{u}\divv\mathbf{u})]
\dot{\mathbf{u}}\cdot\mathbf{n}\mathrm{ds}
\notag\\
&=\sigma\int_{\partial\Omega}(F-\bar{P})_t\dot{\mathbf{u}}\cdot\mathbf{n}\mathrm{ds}
+\sigma\int_{\partial\Omega}\divv[\mathbf{u}(F-\bar{P})]\dot{\mathbf{u}}\cdot\mathbf{n}\mathrm{ds}\notag\\
&=-\sigma\int_{\partial\Omega}(F-\bar{P})_t\mathbf{u}\cdot\nabla\mathbf{n}\cdot \mathbf{u}\mathrm{ds}
+\sigma\int_{\partial\Omega}\divv[\mathbf{u}(F-\bar{P})]\dot{\mathbf{u}}\cdot\mathbf{n}\mathrm{ds}\notag\\
&=-\frac{\mathrm{d}}{\mathrm{d}t}\int_{\partial\Omega}\sigma (F-\bar{P})(\mathbf{u}\cdot\nabla\mathbf{n}\cdot \mathbf{u})\mathrm{ds}+\sigma'\int_{\partial\Omega}(F-\bar{P})(\mathbf{u}\cdot\nabla \mathbf{n}\cdot \mathbf{u})\mathrm{ds}+\sigma\int_{\partial\Omega}(F-\bar{P})(\dot{\mathbf{u}}\cdot\nabla \mathbf{n}\cdot \mathbf{u})\mathrm{ds}\notag\\
&\quad+\sigma\int_{\partial\Omega}(F-\bar{P})(\mathbf{u}\cdot\nabla \mathbf{n}\cdot \dot{\mathbf{u}})\mathrm{ds}-\sigma\int_{\partial\Omega}(F-\bar{P})(\mathbf{u}\cdot\nabla\mathbf{u}\cdot\nabla \mathbf{n}\cdot \mathbf{u})\mathrm{ds}\notag\\
&\quad-\sigma\int_{\partial\Omega}(F-\bar{P})(\mathbf{u}\cdot\nabla \mathbf{n}\cdot (\mathbf{u}\cdot\nabla)\mathbf{u})\mathrm{ds}+\sigma\int_{\partial\Omega}\divv[\mathbf{u}(F-\bar{P})]\dot{\mathbf{u}}\cdot\mathbf{n}\mathrm{ds}\notag\\
&\leq
-\frac{\mathrm{d}}{\mathrm{d}t}\int_{\partial\Omega}\sigma (F-\bar{P})(\mathbf{u}\cdot\nabla\mathbf{n}\cdot \mathbf{u})\mathrm{ds}
+\sigma\int_{\partial\Omega}\divv[\mathbf{u}(F-\bar{P})]\dot{\mathbf{u}}\cdot\mathbf{n}\mathrm{ds}
+\frac{\tilde{C}\sigma}{16}\|\nabla\dot{\mathbf{u}}\|_{L^2}^2
\notag\\&\quad
+\frac{C}{(2\mu+\lambda)^2}\|P-\Bar{P}\|_{L^4}^4
+C\big(1+\|\nabla\mathbf{u}\|_{L^{2}}^{2}\big)\big(1+\sigma\|\sqrt{\rho}\dot{\mathbf{u}}\|_{L^{2}}^{2}
+\|\nabla\mathbf{u}\|_{L^{2}}^{2}\big)\big(\|\sqrt{\rho}\dot{\mathbf{u}}\|_{L^{2}}^{2}
+\|\nabla\mathbf{u}\|_{L^{2}}^{2}\big),
\end{align}
where we have used the following estimates
\begin{align*}
 &\left|\int_{\partial\Omega}(F-\bar{P})[(\mathbf{u}\cdot\nabla \mathbf{n}\cdot \mathbf{u})+
 (\dot{\mathbf{u}}\cdot\nabla \mathbf{n}\cdot \mathbf{u}+\mathbf{u}\cdot\nabla \mathbf{n}\cdot \dot{\mathbf{u}})]\mathrm{ds}\right|\leq C(1+\|\nabla F\|_{L^2})\big(\|\mathbf{u}\|_{H^1}^2+\|\mathbf{u}\|_{H^1}\|\dot{\mathbf{u}}\|_{H^1}\big),\\
&\left|\int_{\partial\Omega}(F-\bar{P})(\mathbf{u}\cdot\nabla)\mathbf{u}\cdot\nabla \mathbf{n}\cdot \mathbf{u}\mathrm{ds}\right|
=\left|\int_{\partial\Omega}(F-\bar{P})(\mathbf{u}\cdot\mathbf{n}^\bot)(\mathbf{n}^\bot\cdot\nabla)\mathbf{u}\cdot\nabla \mathbf{n}\cdot \mathbf{u}\mathrm{ds}\right|\notag\\
&=\left|\int\nabla^\bot\cdot[(F-\bar{P})(\mathbf{u}\cdot\mathbf{n}^\bot)\nabla\mathbf{u}\cdot\nabla\mathbf{n}\cdot \mathbf{u}]\mathrm{d}\mathbf{x}\right|
=\left|\int\nabla u^i\cdot\nabla^\bot(\nabla_i\mathbf{n}\cdot\mathbf{u}(F-\bar{P})
(\mathbf{u}\cdot\mathbf{n}^\bot))\mathrm{d}\mathbf{x}\right|\notag\\
&\leq C(1+\|\nabla F\|_{L^2})\|\nabla\mathbf{u}\|_{L^4}\big(\|\mathbf{u}\|_{L^4}^2+\|\nabla\mathbf{u}\|_{L^4}\|\mathbf{u}\|_{L^4}+\|\mathbf{u}\|_{L^8}^2),\notag\\
&\left|\int_{\partial\Omega}(F-\bar{P})\mathbf{u}\cdot\nabla \mathbf{n}\cdot (\mathbf{u}\cdot\nabla)\mathbf{u}\mathrm{ds}\right|\leq C(1+\|\nabla F\|_{L^2})\|\nabla\mathbf{u}\|_{L^4}
\big(\|\mathbf{u}\|_{L^4}^2+\|\nabla\mathbf{u}\|_{L^4}\|\mathbf{u}\|_{L^4}
+\|\mathbf{u}\|_{L^8}^2\big).
\end{align*}
Moreover, it follows from the divergence theorem that
\begin{align}\label{z3.38}
  &\mathcal{B}_2+\mathcal{B}_5+\sigma\int_{\partial\Omega}\divv[\mathbf{u}(F-\bar{P})]\dot{\mathbf{u}}\cdot\mathbf{n}\mathrm{ds}\notag\\
  &=\sigma\int_{\partial\Omega}[P-(2\mu+\lambda)\divv\mathbf{u}]\dot{\mathbf{u}}\cdot\nabla\mathbf{u}\cdot\mathbf{n}\mathrm{ds}
  +\sigma\int_{\partial\Omega}\divv[\mathbf{u}(F-\bar{P})]\dot{\mathbf{u}}\cdot\mathbf{n}\mathrm{ds}\notag\\
  &=-\sigma\int_{\partial\Omega}(F-\bar{P})\dot{\mathbf{u}}\cdot\nabla\mathbf{u}\cdot\mathbf{n}\mathrm{ds}
  +\sigma\int_{\partial\Omega}\divv[\mathbf{u}(F-\bar{P})]\dot{\mathbf{u}}\cdot\mathbf{n}\mathrm{ds}\notag\\
  &=\sigma\int\divv\big[\dot{\mathbf{u}}\divv\big(\mathbf{u}(F-\bar{P})\big)-(F-\bar{P})\dot{\mathbf{u}}\cdot\nabla\mathbf{u}\big]
  \mathrm{d}\mathbf{x}\notag\\
  &\leq C\sigma(1+\|\nabla F\|_{L^2})\|\nabla\mathbf{u}\|_{L^4}(\|\nabla\dot{\mathbf{u}}\|_{L^2}+\|\dot{\mathbf{u}}\|_{L^4})
  +C\sigma\|\nabla F\|_{L^2}\|\nabla\dot{\mathbf{u}}\|_{L^2}\|\mathbf{u}\|_{L^\infty}\notag\\
  &\quad+\sigma\int(F-\bar{P})(\dot{\mathbf{u}}\cdot\nabla\divv\mathbf{u}-\dot{\mathbf{u}}\cdot\nabla\divv\mathbf{u})
  \mathrm{d}\mathbf{x}+\sigma\int\dot{u}^iu^j\partial_i\partial_jF\mathrm{d}\mathbf{x}
  \notag\\
  &\leq \frac{\tilde{C}\sigma}{16}\|\nabla\dot{\mathbf{u}}\|_{L^2}^2
  +C\sigma\big(1+\|\sqrt{\rho}\dot{\mathbf{u}}\|_{L^{2}}^{2}+\|\nabla\mathbf{u}\|_{L^{2}}^{2}\big)
  \big(\|\sqrt{\rho}\dot{\mathbf{u}}\|_{L^{2}}^{2}+\|\nabla\mathbf{u}\|_{L^{2}}^{2}\big)
 +\frac{C}{(2\mu+\lambda)^2}\|P-\Bar{P}\|_{L^4}^4,
 \end{align}
 where one has used
 \begin{equation*}
\int\dot{u}^iu^j\partial_i\partial_jF\mathrm{d}\mathbf{x}
=-\int\big(\dot{u}^i_ju^j\partial_iF+\dot{u}^iu^j_j\partial_iF\big)\mathrm{d}\mathbf{x}
\leq C\|\nabla F\|_{L^2}\big(\|\nabla\dot{\mathbf{u}}\|_{L^2}\|\mathbf{u}\|_{L^\infty}
+\|\dot{\mathbf{u}}\|_{L^4}\|\nabla\mathbf{u}\|_{L^4}\big).
 \end{equation*}

According to Lemma \ref{dot}, we can choose the constant $\tilde{C}$ such that
\begin{equation*}
\tilde{C}\|\nabla\dot{\mathbf{u}}\|_{L^2}^2\leq
C_1\tilde{C}(\|\divv\dot{\mathbf{u}}\|_{L^2}^2+\|\curl\dot{\mathbf{u}}\|_{L^2}^2+\|\nabla\mathbf{u}\|_{L^4}^4)
\leq \mu(\|\divf\mathbf{u}\|_{L^2}^2+\|\curl\dot{\mathbf{u}}\|_{L^2}^2)   +C\|\nabla\mathbf{u}\|_{L^4}^4.
\end{equation*}
Thus, substituting \eqref{z3.30}--\eqref{z3.38} into \eqref{z3.29}, one derives from \eqref{z3.23}, Lemmas $\ref{Hodge}$, $\ref{l3.1}$, and $\ref{l3.2}$ that
\begin{align}\label{z3.40}
&\frac{\mathrm{d}}{\mathrm{d}t}\left(\frac{1}{2}\int\sigma\rho|\dot{\mathbf{u}}|^2\mathrm{d}\mathbf{x}
+\int\sigma F(\mathcal{P}\mathbf{u})_j^i(\mathcal{P}\mathbf{u})_i^j\mathrm{d}\mathbf{x}
+\int_{\partial\Omega}\sigma (F-\bar{P})(\mathbf{u}\cdot\nabla\mathbf{n}\cdot \mathbf{u})\mathrm{ds}\right)\notag\\ &\quad
+\frac{(2\mu+\lambda)\sigma}{2}\|\divf\mathbf{u}\|_{L^2}^2
+\frac{\mu\sigma}{2}\|\curl\dot{\mathbf{u}}\|_{L^2}^2
\notag\\
&\leq(2+M)^{e^{3D_2(1+C_0)^2}}\big(1+\sigma\|\sqrt{\rho}\dot{\mathbf{u}}\|_{L^{2}}^{2}
+\|\nabla\mathbf{u}\|_{L^{2}}^{2}\big)\big(\|\sqrt{\rho}\dot{\mathbf{u}}\|_{L^{2}}^{2}+
\|\nabla\mathbf{u}\|_{L^{2}}^{2}\big)+\frac{C}{(2\mu+\lambda)^2}\|P-\Bar{P}\|_{L^4}^4,
\end{align}
where we observe that
\begin{align}
\left|\int\sigma F(\mathcal{P}\mathbf{u})_j^i(\mathcal{P}\mathbf{u})_i^j\mathrm{d}\mathbf{x}\right|
+\left|\int_{\partial\Omega}\sigma (F-\bar{P})(\mathbf{u}\cdot\nabla\mathbf{n}\cdot \mathbf{u})\mathrm{ds}\right|
&\leq C\sigma\big(\|\nabla F\|_{L^2}\|\nabla\mathcal{P}\mathbf{u}\|_{L^3}^2
+\|\nabla\mathbf{u}\|_{L^2}^2\|\nabla F\|_{L^2}\big)\notag\\
 &\leq \frac{\sigma}{4}\|\sqrt{\rho}\dot{\mathbf{u}}\|_{L^{2}}^2+
 C\|\nabla\mathbf{u}\|_{L^{2}}^2\big(1+\|\nabla\mathbf{u}\|_{L^{2}}^2\big)^3.\notag
\end{align}

Now we define an auxiliary functional $B_2(t)$ as
\begin{equation*}
  B_2(t)\triangleq\int\frac{\sigma}{2}\rho|\dot{\mathbf{u}}|^2\mathrm{d}\mathbf{x}
  +\int \sigma F(\mathcal{P}\mathbf{u})_j^i(\mathcal{P}\mathbf{u})_i^j\mathrm{d}\mathbf{x}
  +\int_{\partial\Omega}\sigma (F-\bar{P})(\mathbf{u}\cdot\nabla\mathbf{n}\cdot \mathbf{u})\mathrm{ds}+(2+M)^{e^{3D_2(1+C_0)^2}}B_1(t).
\end{equation*}
Then, from Lemma \ref{Hodge}, \eqref{z3.23}, and the definition of $B_1(t)$ in \eqref{z3.18}, one sees that
\begin{equation}\label{z3.41}
  B_2(t)\thicksim \sigma\|\sqrt{\rho}\dot{\mathbf{u}}\|_{L^{2}}^2+B_1(t).
\end{equation}
Setting
\begin{equation*}
  f_2(t)\triangleq 2+B_2(t),~~g_2(t)\triangleq(2+M)^{e^{3D_2(1+C_0)^2}}
  \bigg[\|\sqrt{\rho}\dot{\mathbf{u}}\|_{L^{2}}^2+\|\nabla\mathbf{u}\|_{L^{2}}^2
  +\frac{1}{(2\mu+\lambda)^2}\|P-\Bar{P}\|_{L^4}^4\bigg],
\end{equation*}
we thus get that
\begin{align*}
f'_2(t)\leq g_2(t)f_2(t)
\end{align*}
provided that $\lambda$ satisfies \eqref{lam} with $D\geq 3D_2$.
Therefore, Gronwall's inequality combined with \eqref{z3.41} and Lemmas $\ref{l3.1}\text{--}\ref{l3.3}$ implies that
\begin{equation}\label{z3.43}
  \sup_{0\leq t\leq T}\left(\sigma\|\sqrt{\rho}\dot{\mathbf{u}}\|_{L^{2}}^2\right)\leq
  \exp\bigg\{(2+M)^{e^{\frac{7}{2}D_2(1+C_0)^2}}\bigg\}.
\end{equation}
Integrating \eqref{z3.40} with respect to $t$ over $(0,T)$, one infers from \eqref{z3.43} and Lemmas $\ref{l3.1}\text{--}\ref{l3.3}$ that
\begin{align*}
&\int_0^T\big[(2\mu+\lambda)\sigma\|\divf\mathbf{u}\|_{L^2}^2
+\mu\sigma\|\curl\dot{\mathbf{u}}\|_{L^2}^2\big]\mathrm{d}t\notag\\
&\leq C(1+M)+C(1+C_0)(2+M)^{e^{\frac{7}{2}D_2(1+C_0)^2}}
\exp\bigg\{(2+M)^{e^{\frac{7}{2}D_2(1+C_0)^2}}\bigg\}\notag\\
&\leq\exp\bigg\{(2+M)^{e^{\frac{15}{4}D_2(1+C_0)^2}}\bigg\},
\end{align*}
which together with \eqref{z3.43} leads to the desired \eqref{z3.28}.
\end{proof}

Finally, inspired by \cite{DE97,Hoff02}, we derive the upper bound of the density.
\begin{lemma}\label{l3.5}
Under the assumptions \eqref{z3.1}, it holds that
\begin{align*}
0\leq\rho(\mathbf{x},t)\leq\frac{7}{4}\hat{\rho}~\textit{a.e.}~\mathrm{on}~\Omega\times[0,T]
\end{align*}
provided that $\lambda$ satisfies \eqref{lam} with $D\geq 5D_2$.
\end{lemma}
\begin{proof}
Let $\mathbf{y}\in\Omega$ and define the corresponding particle path $\mathbf{x}(t)$ by
\begin{align*}
\begin{cases}
\mathbf{\dot{x}}(t,\mathbf{y})=\mathbf{u}(\mathbf{x}(t,\mathbf{y}),\mathbf{y)},\\
\mathbf{\dot{x}}(t_0,\mathbf{y})=\mathbf{y}.
\end{cases}
\end{align*}
Assume that there exists $t_1\leq T$ satisfying $\rho(\mathbf{x}(t_1), t_1) = \frac{7}{4}\hat{\rho}$, we take a minimal value of $t_1$ and then choose a maximal value of $t_0<t_1$ such that $\rho(\mathbf{x}(t_0), t_0)=\frac{3}{2}\hat{\rho}$. Thus, $\rho(\mathbf{x}(t),t)\in[\frac{3}{2}\hat{\rho},\frac{7}{4}\hat{\rho}]$ for $t\in[t_0,t_1]$. We divide
the argument into two cases.

\textbf{Case 1:} $t_0<t_1\leq1$. One gets from $\eqref{a1}_1$ and \eqref{z1.4} that
\begin{equation*}
(2\mu+\lambda)\frac{\mathrm{d}}{\mathrm{d}t}\ln\rho(\mathbf{x}(t),t)
+P(\rho(\mathbf{x}(t),t))-\bar{P}=-F(\mathbf{x}(t),t),
\end{equation*}
where $\frac{\mathrm{d}\rho}{\mathrm{d}t}\triangleq \rho_t+\mathbf{u}\cdot\nabla\rho$. Integrating the above equality from $t_0$ to $t_1$ and abbreviating $\rho(\mathbf{x},t)$ by $\rho(t)$ for
convenience, we have that
\begin{equation}\label{z3.44}
\ln\rho(\tau)\big|_{t_0}^{t_1}+\frac{1}{2\mu+\lambda}\int_{t_0}^{t_1}\left(P(\rho(\tau))
-\bar{P}\right)\mathrm{d}\tau=-\frac{1}{2\mu+\lambda}\int_{t_0}^{t_1}F(\mathbf{x}(\tau),\tau)\mathrm{d}\tau.
\end{equation}
It follows from Lemmas $\ref{GN}$, $\ref{dot}$, and $\ref{l3.1}\text{--}\ref{l3.4}$ that
\begin{align}\label{z3.45}
&\int_0^{\sigma(T)}\left\|F(\cdot,t)\right\|_{L^\infty}\mathrm{d}t\leq C\int_0^{\sigma(T)}\left\|F\right\|_{L^2}^{\frac13}\left\|\nabla F\right\|_{L^4}^{\frac23}\mathrm{d}t\notag\\
&\leq C\int_0^{\sigma(T)}\Big((2\mu+\lambda)^{\frac13}
\|\divv\mathbf{u}\|_{L^2}^{\frac13}+\|P-\bar{P}\|_{L^2}^{\frac13}\Big)
\Big(\|\sqrt{\rho}\dot{\mathbf{u}}\|_{L^4}^{\frac23}
+\|\nabla\mathbf{u}\|_{L^4}^{\frac23}\Big)\mathrm{d}t\notag\\
&\leq C\sup_{0\leq t\leq T}\Big((2\mu+\lambda)^{\frac13}\|\divv\mathbf{u}\|_{L^2}^{\frac13}
+\|P-\bar{P}\|_{L^2}^{\frac13}\Big)
\int_{0}^{\sigma(T)}
\Big(
\|\nabla\dot{\mathbf{u}}\|_{L^{2}}^{\frac{2}{3}}+\|\nabla\mathbf{u}\|_{L^2}^{\frac43}
+\|\nabla\mathbf{u}\|_{L^4}^{\frac23}\Big) \mathrm{d}t\notag\\
&\leq C\Big((2\mu+\lambda)^{\frac{1}{6}}(2+M)^{\frac{1}{6}\mathrm{e}^{{2D_{2}
(1+C_{0})^{2}}}}+1\Big)
\int_{0}^{\sigma(T)}
\Big(\|\nabla\dot{\mathbf{u}}
\|_{L^{2}}^{\frac{2}{3}}+\|\sqrt{\rho}\dot{\mathbf{u}}
\|_{L^{2}}^{\frac{1}{3}}\|\nabla\dot{\mathbf{u}}
\|_{L^{2}}^{\frac{1}{3}}\Big) \mathrm{d}t\notag\\
&\leq(2\mu+\lambda)^{\frac16}(2+M)^{\frac12\mathrm{e}^{2D_2(1+C_0)^2}}\bigg[\int_0^{\sigma(T)}\|\sqrt{\rho}\dot{\mathbf{u}}\|_{L^2}^{\frac23}\mathrm{d}t
+\int_0^{\sigma(T)}\big(t\|\divf\mathbf{u}\|_{L^2}^2+t\|\curl\dot{\mathbf{u}}\|_{L^2}^2\big)^{\frac13}t^{-\frac13}\mathrm{d}t\bigg]\notag\\
&\leq(2\mu+\lambda)^{\frac{1}{6}}(2+M)^{\frac{1}{2}\mathrm{e}^{2D_{2}(1+C_{0})^{2}}} \Bigg[\bigg(\int_{0}^{\sigma(T)}\|\sqrt{\rho}\dot{\mathbf{u}}\|_{L^{2}}^{2}\mathrm{d}t\bigg)^{\frac{1}{3}}\bigg(\int_{0}^{\sigma(T)}1\mathrm{d}t\bigg)^{\frac{2}{3}}
\notag\\&\quad+\bigg(\int_0^{\sigma(T)}\big(t\|\divf\mathbf{u}\|_{L^2}^2+t\|\curl\dot{\mathbf{u}}\|_{L^2}^2\big)
\mathrm{d}t\bigg)^{\frac13}\bigg(\int_0^{\sigma(T)}
t^{-\frac12}\mathrm{d}t\bigg)^{\frac23}\Bigg]\notag\\
&\leq(2\mu+\lambda)^{\frac{1}{6}}\exp\Big\{(2+M)^{\mathrm{e}^{\frac{9}{2}D_{2}(1+C_{0})^{2}}}\Big\}.
\end{align}
Since $\rho(t)$ takes values in $[\frac{3}{2}\hat{\rho},\frac{7}{4}\hat{\rho}]\subset [\hat{\rho},2\hat{\rho}]$ and $P(\rho)$ is increasing on $[0,\infty)$, we obtain, after substituting $\eqref{z3.45}$ into $\eqref{z3.44}$, that
\begin{equation*}
    \ln\left(\frac{7}{4}\hat{\rho}\right)-\ln\left(\frac{3}{2}\hat{\rho}\right)+\frac{1}{2\mu+\lambda}\int_{t_0}^{t_1}\left(P(\rho(\tau))
-\bar{P}\right)\mathrm{d}\tau
    \leq\frac{1}{\left(2\mu+\lambda\right)^{\frac{5}{6}}}
    \exp\bigg\{\left(2+M\right)^{\mathrm{e}^{\frac{19}{4}D_{2}(1+C_{0})^{2}}}\bigg\},
\end{equation*}
which is impossible if $\lambda$ satisfies \eqref{lam} with $D\geq 5D_2$. Therefore, we conclude that there is no time $t_1$ such that $\rho(\mathbf{x}(t_1), t_1) = \frac{7}{4}\hat{\rho}$. Since $\mathbf{y}\in\Omega$ is arbitrary, it follows that $\rho<\frac{7}{4}\hat{\rho}$ on $\Omega\times[0,T]$.

\textbf{Case 2:} $t_1>1$. From $\eqref{a1}_1$ and \eqref{z1.4}, one sees that
\begin{equation*}
    \frac{\mathrm{d}}{\mathrm{d}t}\rho(t)+\frac{1}{2\mu+\lambda}\rho(t)(P(\rho(t))-\bar{P})=-\frac{1}{2\mu+\lambda}\rho(t)F(\mathbf{x}(t),t).
\end{equation*}
Multiplying the above equality by $|\rho(t)|\rho(t)$, we have
\begin{equation}\label{z3.46}
    \frac{1}{3}\frac{\mathrm{d}}{\mathrm{d}t}\left|\rho(t)\right|^3+\frac{1}{2\mu+\lambda}|\rho(t)|^3(P(\rho(t))-\bar{P})
    =-\frac{1}{2\mu+\lambda}|\rho(t)|^3F(\mathbf{x}(t),t).
\end{equation}
If $\rho(t)$ takes values in $[\frac{3}{2}\hat{\rho},\frac{7}{4}\hat{\rho}]$, by integrating \eqref{z3.46} from $t_0$ to $t_1$, we obtain from Young's inequality that
\begin{align*}
\hat{\rho}^{3}
&\leq\frac{C}{2\mu+\lambda}\int_{0}^{1}\|F(\cdot,t)\|_{L^\infty}\mathrm{d}t+
\frac{C}{2\mu+\lambda}\int_{1}^{T}\|F(\cdot,t)\|_{L^\infty}^3\mathrm{d}t\notag\\
&\leq\frac{1}{(2\mu+\lambda)^{\frac{5}{6}}}
\exp\bigg\{(2+M)^{\mathrm{e}^{\frac{9}{2}D_{2}(1+C_{0})^{2}}}\bigg\}
+\frac{C}{2\mu+\lambda}\int_{1}^{T}\|F(\cdot,t)\|_{L^{2}}\|\nabla F(\cdot,t)\|_{L^{4}}^{2}\mathrm{d}t\notag\\
&\leq\frac{1}{(2\mu+\lambda)^{\frac{5}{6}}}
\exp\bigg\{(2+M)^{\mathrm{e}^{\frac{9}{2}D_{2}(1+C_{0})^{2}}}\bigg\}
+\frac{(2+M)^{\mathrm{e}^{2D_2(1+C_0)^2}}}{(2\mu+\lambda)^{\frac{1}{2}}}\int_{0}^{T}
\|\nabla\dot{\mathbf{u}}\|_{L^2}\big(\|\sqrt{\rho}\dot{\mathbf{u}}\|_{L^2}
+\|\nabla\dot{\mathbf{u}}\|_{L^2}\big)\mathrm{d}t\notag\\
&\leq\frac{1}{(2\mu+\lambda)^{\frac{5}{6}}}
\exp\bigg\{(2+M)^{\mathrm{e}^{\frac{9}{2}D_{2}(1+C_{0})^{2}}}\bigg\}
+\frac{1}{(2\mu+\lambda)^{\frac{1}{2}}}(2+M)^{\mathrm{e}^{2D_2(1+C_0)^2}}
\exp\bigg\{(2+M)^{\mathrm{e}^{\frac{9}{2}D_{2}(1+C_{0})^{2}}}\bigg\}\notag\\
&\leq\frac{1}{(2\mu+\lambda)^{\frac12}}
\exp\bigg\{(2+M)^{\mathrm{e}^{\frac{19}{4}D_{2}(1+C_{0})^{2}}}\bigg\},
\end{align*}
which is impossible if $\lambda$ satisfies \eqref{lam} with $D\geq 5D_2$. Hence, we conclude that there is no time $t_1$ such that $\rho(\mathbf{x}(t_1), t_1) = \frac{7}{4}\hat{\rho}$. Since $\mathbf{y}\in\Omega$ is arbitrary, it follows that $\rho<\frac{7}{4}\hat{\rho}$ on $\Omega\times[0,T]$.
\end{proof}

Now we are ready to prove Proposition $\ref{p3.1}$.
\begin{proof}[Proof of Proposition \ref{p3.1}.]
Proposition \ref{p3.1} follows from the Lemmas $\ref{l3.2}$--$\ref{l3.5}$ provided that $\lambda$ satisfies \eqref{lam} with $D\geq 5D_2$.
\end{proof}

\section{Proof of Theorem \ref{t1.1}}\label{sec4}

In this section we use the \textit{a priori} estimates established in Section $\ref{sec3}$ to complete proof of Theorem \ref{t1.1}.
\begin{proof}[Proof of Theorem \ref{t1.1}.]
Let $(\rho_0, \mathbf{u}_0)$ be initial data as described in the theorem.
For $\varepsilon>0$, let $j_\varepsilon=j_\varepsilon(\mathbf{x})$ be the standard mollifier, define the approximate initial data $(\rho_0^\varepsilon, \mathbf{u}_0^\varepsilon)$:
\begin{align*}
\rho_0^\varepsilon&=[J_\varepsilon\ast(\rho_0\mathbf{1}_\Omega)]\mathbf{1}_\Omega+\varepsilon,
\end{align*}
and $\mathbf{u}_0^\varepsilon$ is the unique smooth solution to the following elliptic equation
\begin{equation*}
\begin{cases}
\Delta \mathbf{u}_0^\varepsilon = \Delta(J_\varepsilon\ast\mathbf{u}_0), & \mathbf{x}\in \Omega, \\
\mathbf{u}_0^\varepsilon \cdot \mathbf{n} = 0, \,\ \curl \mathbf{u}_0^\varepsilon = 0, & \mathbf{x}\in\partial \Omega.
\end{cases}
\end{equation*}
Then we have
\begin{align*}
\rho_0^\varepsilon\in H^2,\ \
\inf_{\mathbf{x}\in\Omega}\{\rho_0^\varepsilon(\mathbf{x})\}\geq\varepsilon, \ \ \mathbf{u}_0^\varepsilon\in  H^2_\omega.
\end{align*}
By Proposition $\ref{p3.1}$, it holds that, for $\varepsilon$ being suitably small,
\begin{align*}
0\leq\rho^\varepsilon(\mathbf{x},t)\leq\frac{7}{4}\hat{\rho}~\textit{a.e.}~\mathrm{on}~\Omega\times[0,T]
\end{align*}
provided that $\lambda$ satisfies \eqref{lam}. Thus, Lemma $\ref{l2.1}$ implies the global existence and uniqueness of strong solutions $(\rho^\varepsilon,\mathbf{u}^\varepsilon)$ to \eqref{a1} and \eqref{a3} with the smooth initial data $(\rho_0^\varepsilon,\mathbf{u}_0^\varepsilon)$.

Fix $\mathbf{x}\in\overline{\Omega}$ and let $B_R$ be a ball of radius $R$ centered at $\mathbf{x}$. Then, for $t\geq\tau>0$, it follows from Lemmas $\ref{E0}$, $\ref{l3.1}\text{--}\ref{l3.4}$, and Sobolev's inequality that
\begin{align}
\langle\mathbf{u}^\varepsilon(\cdot,t)\rangle^{\frac12}_{\overline{\Omega}}&\leq C\big(1+\|\nabla\mathbf{u}^\varepsilon\|_{L^4}\big)\notag\\
&\leq C\|\nabla\mathbf{u}^\varepsilon\|_{L^2}^{\frac12}
\big\|\sqrt{\rho^\varepsilon}\dot{\mathbf{u}}^\varepsilon\big\|_{L^2}
^{\frac12}+\frac{C}{2\mu+\lambda}
\big\|\sqrt{\rho^\varepsilon}\dot{\mathbf{u}}^\varepsilon\big\|_{L^2}^{\frac12}
\big\|P(\rho^\varepsilon)-\overline{P(\rho^\varepsilon)}\big\|_{L^2}^\frac{1}{2}\notag\\
&\quad+\frac{C}{2\mu+\lambda}
\big\|P(\rho^\varepsilon)-\overline{P(\rho^\varepsilon)}\big\|_{L^4}
+C\|\nabla \mathbf{u}^\varepsilon\|_{L^2}+C
\notag\\
&\leq C(\tau).\notag
\end{align}
Note that
\begin{align}
\left|\mathbf{u}^\varepsilon(\mathbf{x},t)-\frac{1}{|B_R\cap\Omega|}
\int_{B_R\cap\Omega}\mathbf{u}^\varepsilon(\mathbf{y},t)\mathrm{d}\mathbf{y}\right|
&=\left|\frac{1}{|B_R\cap\Omega|}
\int_{B_R\cap\Omega}\left(\mathbf{u}^\varepsilon(\mathbf{x},t)
-\mathbf{u}^\varepsilon(\mathbf{y},t)\right)\mathrm{d}\mathbf{y}\right|\notag\\
&\leq\frac{1}{|B_R\cap\Omega|}C(\tau)\int_{B_R\cap\Omega}
|\mathbf{x}-\mathbf{y}|^{\frac12}\mathrm{d}\mathbf{y}
\notag\\
&\leq C(\tau)R^{\frac12}.\notag
\end{align}
For $0<\tau\leq t_1<t_2<\infty$, we deduce that
\begin{align}
|\mathbf{u}^\varepsilon(\mathbf{x},t_2)-\mathbf{u}^\varepsilon(\mathbf{x},t_1)|
&\leq\frac{1}{|B_R\cap\Omega|}\int_{t_{1}}^{t_{2}}
\int_{B_R\cap\Omega}|\mathbf{u}_{t}^{\varepsilon}(\mathbf{y},t)
|\mathrm{d}\mathbf{y}\mathrm{d}t+C(\tau)R^{\frac12}\notag\\
&\leq CR^{-1}|t_{2}-t_{1}|^{\frac{1}{2}}\left(\int_{t_{1}}^{t_{2}}
\int\left|\mathbf{u}_{t}^{\varepsilon}(\mathbf{y},t)\right|^{2}
\mathrm{d}\mathbf{y}\mathrm{d}t\right)^{\frac{1}{2}}+C(\tau)R^{\frac12}\notag\\
&\leq CR^{-1}|t_{2}-t_{1}|^{\frac{1}{2}}\left(\int_{t_{1}}^{t_{2}}
\int\big(|\dot{\mathbf{u}}^{\varepsilon}|^{2}
+|\mathbf{u}^{\varepsilon}|^{2}|\nabla\mathbf{u}^{\varepsilon}|^{2}\big)
\mathrm{d}\mathbf{y}\mathrm{d}t\right)^{\frac{1}{2}}
+C(\tau)R^{\frac12}\notag\\
&\leq C(\tau)R^{-1}|t_{2}-t_{1}|^{\frac{1}{2}}+C(\tau)R^{\frac12},\notag
\end{align}
due to
\begin{align}
\int_{t_1}^{t_2}\int|\mathbf{u}^\varepsilon|^2
|\nabla\mathbf{u}^\varepsilon|^2\mathrm{d}\mathbf{x}\mathrm{d}t
&\leq C\sup_{t_1\leq t\leq t_2}\|\mathbf{u}^\varepsilon\|_{L^\infty}^2\int_{t_1}^{t_2}
\int|\nabla\mathbf{u}^\varepsilon|^2\mathrm{d}\mathbf{x}\mathrm{d}t\notag\\
&\leq C\sup_{t_1\leq t\leq t_2}\|\mathbf{u}^\varepsilon\|_{L^2}^{\frac{2}{3}}
\|\nabla\mathbf{u}^\varepsilon\|_{L^4}^{\frac{4}{3}}
\int_{t_1}^{t_2}\int|\nabla\mathbf{u}^\varepsilon|^2\mathrm{d}\mathbf{x}\mathrm{d}t\leq C(\tau).\notag
\end{align}
Choosing $R=|t_2-t_1|^{\frac13}$, one can see that
\begin{equation*}
|\mathbf{u}^\varepsilon(\mathbf{x},t_2)-\mathbf{u}^\varepsilon(\mathbf{x},t_1)|
\leq C(\tau)|t_{2}-t_{1}|^{\frac{1}{6}},~~\text{for}~~0<\tau\leq t_1<t_2<\infty,
\end{equation*}
which implies that $\{\mathbf{u}^\varepsilon\}$ is uniformly H\"older continuous away from $t=0$.

For any fixed $\tau$ and $T$ with $0<\tau<T<\infty$, it follows from Ascoli--Arzel\`{a} theorem that there is a subsequence $\varepsilon_k\rightarrow0$ satisfying
\begin{equation}\label{z4.1}
\mathbf{u}^{\varepsilon_k}\rightarrow \mathbf{u}~~\mathrm{uniformly}~\mathrm{on}~\mathrm{compact}~\mathrm{sets} ~\mathrm{in}~\Omega\times(0,\infty).
\end{equation}
Moreover, by the standard compactness arguments as those in \cite{Hoff05,EF01,PL98}, we can extract a further subsequence $\varepsilon_{k'}\rightarrow0$ satisfying
\begin{equation}\label{z4.2}
  \rho^{\varepsilon_{k'}}\rightarrow\rho~~\mathrm{strongly}~\mathrm{in}~L^p(\Omega),~~\mathrm{for}~
  \mathrm{any}~p\in[1,\infty)~\mathrm{and}~t\geq0.
\end{equation}
Therefore, passing to the limit of $\varepsilon_{k'}\rightarrow0$, in view of \eqref{z4.1} and \eqref{z4.2}, we deduce that the limit function $(\rho,\mathbf{u})$ is indeed a weak solution of the initial-boundary value problem \eqref{a1}--\eqref{a3} in the sense of Definition $\ref{d1.1}$ and satisfies \eqref{reg}.
\end{proof}

\section{Proof of Theorem \ref{t1.2}}\label{sec5}
This section is devoted to the incompressible limit of \eqref{a1}--\eqref{a3} as the bulk viscosity tends to infinity.

\begin{proof}[Proof of Theorem \ref{t1.2}.]
Let $\{(\rho^\lambda,{\bf u}^\lambda)\}$ be the family of solutions to the problem \eqref{a1}--\eqref{a3} from Theorem \ref{t1.1}. Applying \eqref{reg} and performing a similar argument as that in \eqref{z4.1}, then there is a subsequence $\{(\rho^{\lambda_{k}},{\bf u}^{\lambda_{k}})\}$ such that
\begin{align}
{\bf u}^{\lambda_{k}}&\rightarrow {\bf v}
~~\text{uniformly on compact sets  in}~\Omega\times(0,\infty),\notag\\
\rho^{\lambda_{k}}&\rightarrow \varrho\ \    \text{weakly in}\ \ L^p(\Omega),\ \ \text{for any}\ p\in[1,\infty)\ \text{and}\ {t\ge 0},\label{5.1}\\
\rho^{\lambda_{k}}&\rightarrow \varrho\ \  \text{weakly* in}\ \ L^\infty(\Omega),\ \ \text{for any}\ {t\ge 0},\notag\\
\divv{\bf u}^{\lambda_{k}}&\rightarrow 0~~\text{strongly  in}~L^2(\Omega\times(0,\infty)).\notag
\end{align}
Hence, we conclude that $\divv{\bf v}=0$ and $(\varrho,{\bf v})$ satisfies \eqref{1.13} and \eqref{1.14} for all $C^1$ test functions $(\phi,\boldsymbol\psi)$ just as in Definition \ref{d1.2}, with $\divv\boldsymbol\psi=0$ on $\Omega\times[0,\infty)$. Moreover, $(\varrho,{\bf v})$ satisfies
\begin{equation}\label{5.2}
 0\leq\varrho({\bf x},t)\leq 2 \hat\rho\ \ \text{a.e. on} \
\Omega\times[0,\infty),
\end{equation}
\begin{equation}\label{5.3}
\sup\limits_{t\ge 0}\big(\|\sqrt{\varrho}{\bf v}\|_{L^2}^2+\|{\bf v}\|_{H^1}^2+\sigma\|\nabla^2{\bf v}\|_{L^2}^2\big)+\int_0^\infty\big(\mu\|{\bf v}\|_{H^1}^2+\|\nabla^2{\bf v}\|_{L^2}^2
\big)\mathrm{d}\tau\le C(C_0,M).
\end{equation}

It remains to show \eqref{1.11} and \eqref{1.12}. According to the mass equation $\eqref{a1}_1$, it holds that
\begin{equation*}
\partial_t(\rho^{\epsilon,\lambda}-\rho_0^\epsilon)^2+{\bf u}^{\epsilon,\lambda}\cdot\nabla(\rho^{\epsilon,\lambda}-\rho_0^\epsilon)^2
+2\rho^{\epsilon,\lambda}(\rho^{\epsilon,\lambda}-\rho_0^\epsilon)\divv{\bf u}^{\epsilon,\lambda}=0.
\end{equation*}
Integrating the above equality over $\Omega\times(0,t)$, we obtain that
\begin{align*}
\big\|(\rho^{\epsilon,\lambda}-\rho_0^\epsilon)(\cdot,t)\big\|_{L^2}^2
& =\int_0^t\int(\rho^{\epsilon,\lambda}-\rho_0^\epsilon)^2\divv{\bf u}^{\epsilon,\lambda}\mathrm{d}{\bf x}\mathrm{d}\tau-2\int_0^t\int\rho^{\epsilon,\lambda}(\rho^{\epsilon,\lambda}
-\rho_0^\epsilon)\divv{\bf u}^{\epsilon,\lambda}\mathrm{d}{\bf x}\mathrm{d}\tau
\\ & \le C\bigg(\int_0^t\big\|\rho^{\epsilon,\lambda}-\rho_0^\epsilon\big\|_{L^4}^4
\mathrm{d}\tau\bigg)^\frac{1}{2}\bigg(\int_0^t\big\|\divv{\bf u}^{\epsilon,\lambda}\big\|_{L^2}^2\mathrm{d}\tau\bigg)^\frac{1}{2}
\\ & \quad +C\sup\limits_{t\ge 0}\big\|\rho^{\epsilon,\lambda}(\cdot,t)\big\|_{L^\infty} \bigg(\int_0^t\big\|\rho^{\epsilon,\lambda}-\rho_0^\epsilon\big\|_{L^2}^2
\mathrm{d}\tau\bigg)^\frac{1}{2}\bigg(\int_0^t\big\|\divv{\bf u}^{\epsilon,\lambda}\big\|_{L^2}^2\mathrm{d}\tau\bigg)^\frac{1}{2}
\\ & \le C(t)\lambda^{-\frac{1}{2}},
\end{align*}
which together with \eqref{z4.2} yields that
\begin{equation*}
\|(\rho^\lambda-\rho_0)(\cdot,t)\|_{L^2}^2
=\lim\limits_{\epsilon_k\rightarrow0}
\big\|(\rho^{\epsilon_k,\lambda}-\rho_0^{\epsilon_k})(\cdot,t)\big\|_{L^2}^2
\le C(t)\lambda^{-\frac{1}{2}}.
\end{equation*}
Thus, one has that
\begin{equation}\label{5.4}
\lim\limits_{\lambda\rightarrow\infty}\|(\rho^\lambda-\rho_0)(\cdot,t)\|_{L^2}
=0,\ \ \text{for any}\ t\ge 0.
\end{equation}

Next, using the mollifier $j_\epsilon$ as test functions in \eqref{1.13}, one infers from \eqref{5.3} that, for any compact set $K\subset\Omega$,
\begin{equation}\nonumber\partial_t[\varrho]_\epsilon+{\bf v}\cdot\nabla [\varrho]_\epsilon=\divv\left([\varrho]_\epsilon {\bf v}\right)-\divv[\rho{\bf v}]_\epsilon~~ \text{a.e. on}\ K\times(0,\infty),
\end{equation}
and furthermore,
\begin{equation*}
\partial_t\left([\varrho]_\epsilon-\rho_0\right)^2
+{\bf v}\cdot\nabla \left([\varrho]_\epsilon-\rho_0\right)^2=2\left([\varrho]_\epsilon-\rho_0\right)\big(\divv\left([\varrho-\rho_0]_\epsilon {\bf v}\right)-\divv\left[(\varrho-\rho_0){\bf v}\right]_\epsilon\big)~~ \text{a.e. on}~  K\times(0,\infty).
\end{equation*}
Integrating the above equality over $K\times(0,t)$, we have that
\begin{align}\label{5.5}
&\big|\|([\varrho]_{\epsilon}-\rho_0)(\cdot,t)\|_{L^2(K)}^2- \|[\rho_0]_{\epsilon}-\rho_0\|_{L^2(K)}^2\big|\notag\\
 & \le C\sup\limits_{t\ge 0}\|[\varrho]_\epsilon-\rho_0\|_{L^\infty}\int_0^t
 \|\divv\left([\varrho-\rho_0]_\epsilon {\bf v}\right)-\divv\left[(\varrho-\rho_0){\bf v}\right]_\epsilon\|_{L^1(K)}\mathrm{d}\tau
\notag \\
& \le C\int_0^t\|\divv\left([\varrho-\rho_0]_\epsilon {\bf v}\right)-\divv\left[(\varrho-\rho_0){\bf v}\right]_\epsilon\|_{L^1(K)}\mathrm{d}\tau.
\end{align}
According to Lemma \ref{lcom}, it holds that
\begin{equation*}
\|\divv\left([\varrho-\rho_0]_\epsilon {\bf v}\right)-\divv\left[(\varrho-\rho_0){\bf v}\right]_\epsilon\|_{L^1(K)}\le C(K)\|\varrho-\rho_0\|_{L^2(\Omega)}\|{\bf v}\|_{W^{1,2}(\Omega)}\in L^1(0,T),\ \ \text{for any}\ T>0.
\end{equation*}
This together with Lebesgue's dominated convergence theorem and Lemma \ref{lcom} leads to
\begin{align}\label{5.6}
&\lim\limits_{\epsilon\rightarrow0}\int_0^t
\|\divv\left([\varrho-\rho_0]_\epsilon {\bf v}\right)-\divv\left[(\varrho-\rho_0){\bf v}\right]_\epsilon\|_{L^1(K)}\mathrm{d}\tau \notag \\
& =\int_0^t\lim\limits_{\epsilon\rightarrow0}\|\divv\left([\varrho-\rho_0]_\epsilon {\bf v}\right)-\divv\left[(\varrho-\rho_0){\bf v}\right]_\epsilon\|_{L^1(K)}\mathrm{d}\tau=0.
\end{align}
Substituting \eqref{5.6} into \eqref{5.5}, we deduce that
\begin{align}\label{5.7}
\|(\varrho-\rho_0)(\cdot,t)\|_{L^2(K)}^2= \lim\limits_{\epsilon\rightarrow0} \big|\| ([\varrho]_{\epsilon}-\rho_0)(\cdot,t)\|_{L^2(K)}^2- \|[\rho_0]_{\epsilon}-\rho_0\|_{L^2(K)}^2\big|=0,
\end{align}
which implies \eqref{1.12}.

Finally, combining \eqref{5.4} and \eqref{5.7}, it follows that
\begin{align*}
\lim\limits_{\lambda\rightarrow\infty}\|(\rho^\lambda-\varrho)(\cdot,t)\|_{L^2(K)}
=0,\ \ \text{for any compact set} \ K \subset \Omega \text{ and any} \ t\ge 0,
\end{align*}
which along with \eqref{5.1} yields \eqref{1.11}. Consequently, $(\varrho,{\bf v})$ is a global weak solution to the inhomogeneous incompressible Navier--Stokes equations \eqref{a5} in the sense of Definition \ref{d1.2}.
\end{proof}

\section*{Conflict of interests}
The authors declare that they have no conflict of interests.

\section*{Data availability}
No data was used for the research described in the article.

\end{document}